\documentclass{amsart}%
\usepackage{palatino, mathpazo}
\usepackage{amsfonts}
\usepackage{amsmath}
\usepackage{amssymb,latexsym}
\usepackage{graphicx}
\usepackage[mathscr]{eucal}
\usepackage{amssymb}%
\providecommand{\U}[1]{\protect \rule{.1in}{.1in}}

\newtheorem{theorem}{Theorem}[section]

\newtheorem{corollary}[theorem]{Corollary}

\newtheorem{lemma}[theorem]{Lemma}
\newtheorem{proposition}[theorem]{Proposition}
\newtheorem{Theorem}{Theorem}

\theoremstyle{remark}
\newtheorem{remark}[theorem]{Remark}

\numberwithin{equation}{section}

\def\mod{\ \mathrm{mod}\ }

\def\Im{\mathrm{Im\,}}

\def\SM#1#2#3#4{\left(\begin{smallmatrix}#1&#2\\#3&#4\end{smallmatrix}
  \right)}

\begin{document}
\title[]{Monodromy of a generalized Lam\'{e} equation of third order}
\author{Zhijie Chen}
\address{Department of Mathematical Sciences, Yau Mathematical Sciences Center,
Tsinghua University, Beijing, 100084, China}
\email{zjchen2016@tsinghua.edu.cn}
\author{Chang-Shou Lin}
\address{Department of Mathematics, Taiwan University, Taipei 10617, Taiwan }
\email{cslin@math.ntu.edu.tw}

\begin{abstract}
We study the monodromy of the following third order linear differential equation
\[y'''(z)-(\alpha\wp(z;\tau)+B)y'(z)+\beta\wp'(z;\tau)y(z)=0,
\]
where $B\in\mathbb{C}$ is a parameter, $\wp(z;\tau)$ is the Weierstrass $\wp$-function with periods $1$ and $\tau$, and $\alpha,\beta$ are constants such that
the local exponents at the singularity $0$ are three distinct integers, which can always be written as $-n-l, 1-l, n+2l+2$ after a dual transformation, where $n,l\in\mathbb{N}$. This ODE can be seen as the third order version of the well-known Lam\'{e} equation $y''(z)-(m(m+1)\wp(z;\tau)+B)y(z)=0$. We say that the monodromy is unitary if the monodromy group is conjugate to a subgroup of the unitary group. We show that 
\begin{itemize}
\item[(i)] if $n, l$ are both odd, then the monodromy can not be unitary; 

\item[(ii)] if $n$ is odd and $l$ is even, then there exist finite values of $B$ such that the monodromy is the Klein four-group and hence unitary; 

\item[(iii)] if $n$ is even, then whether there exists $B$ such that the monodromy is unitary depends on the choice of the period $\tau$. 
\end{itemize}
The methods of studying the second order Lam\'{e} equation can not work  here, and
we need to develop different approaches to treat these different cases separately. These results have interesting applications to the integrable $SU(3)$ Toda system in another work (Chen-Lin, J. Differ. Geom. to appear).
\end{abstract}

\maketitle


\section{Introduction}

In this paper, we study the monodromy of
a third order linear differential equation of the following form
\begin{equation}\label{3ode}
y'''(z)-(\alpha \wp(z;\tau)+B)y'(z)+\beta\wp'(z;\tau)y(z)=0,
\end{equation}
where $\wp(z)=\wp(z;\tau)$ is the Weierstrass $\wp$-function with periods $\omega_1=1$ and $\omega_2=\tau\in\mathbb{H}:=\{\tau\in\mathbb{C}|\Im\tau>0\}$, $\alpha,\beta$ are constants such that
the local exponents at the singularity $0$ are three distinct integers, denoted by $\varsigma_1<\varsigma_2<\varsigma_3$, and $B\in\mathbb{C}$ is a parameter.
The ODE (\ref{3ode}) arises from the subjects of algebraically integrable differential operators on the elliptic curve $E_{\tau}:=\mathbb{C}/(\mathbb{Z}+\tau\mathbb{Z})$ (see \cite{ER}) and the $SU(3)$ Toda system on $E_{\tau}$ (see \cite{CL-noncritical}).
It can also be seen as the third order generalization of the well-known Lam\'{e} equation
\begin{equation}\label{new-lm}
y''-[m(m+1)\wp(z)+B]y=0, \quad m\in\mathbb{N}_{\geq 1}.
\end{equation}
The Lam\'{e} equation \eqref{new-lm} has been widely studied in the literature, see e.g. \cite{CLW,Gesztesy-Weikard,Gesztesy-Weikard1,CLW2,Takemura1,Takemura3,Whittaker-Watson} and references therein, and the monodromy of \eqref{new-lm} has been well understood; see Theroems \ref{thmA}-\ref{thmB} below.
However, as far as we know, there are few results for (\ref{3ode}) in the literature.

Let us recall some basic notions for linear ODEs with complex variable (see e.g. \cite[Chapter 1]{GP}). The first one is the local exponents at a singularity. Clearly (\ref{3ode}) has only one singularity at $0$ in $E_{\tau}$. Suppose (\ref{3ode}) has a local solution of the following form near the singularity $0$
\begin{equation}\label{ya}y_a(z)=z^{a}\sum_{j=0}^{\infty}c_j z^j=c_0z^{a}(1+O(z))\;\text{ with }c_0\neq 0,\end{equation}
then inserting this expression in (\ref{3ode}) easily leads to
\begin{equation} \label{a3}
a^3-3a^2+(2-\alpha)a-2\beta=0,
\end{equation}
the zeros of which, denoted by $\varsigma_1, \varsigma_2, \varsigma_3$,
are called the \emph{local exponents} of (\ref{3ode}) at $0$.
Thus, $\alpha, \beta$ can be determined by
\[\alpha=2-\varsigma_1\varsigma_2-\varsigma_1\varsigma_3-\varsigma_2\varsigma_3,\quad \beta=\varsigma_1\varsigma_2\varsigma_3/2.\]

\begin{remark}
In this paper, we always assume that the local exponents are three distinct integers, i.e. $\varsigma_1<\varsigma_2<\varsigma_3$ and $\varsigma_j\in\mathbb Z$. Thus there are two nonnegative integers $n_1, n_2$ such that
\[\varsigma_2=\varsigma_1+n_1+1,\quad \varsigma_3=\varsigma_2+n_2+1.\]
Since \eqref{a3} gives $\sum \varsigma_j=3$, we immediately obtain $\varsigma_1=-\frac{2n_1+n_2}{3}\in\mathbb{Z}$, so $3\big| (n_2-n_1)$. Furthermore, we will explain in Remark \ref{rmk-dual} that the case $n_2<n_1$ can be reduced to the case $\tilde{n}_1=n_2<\tilde{n}_2=n_1$ by considering the dual equation. Therefore, we only need to study the case $n_1\leq n_2$. In conclusion, we always assume in the sequel that
\begin{equation}\label{n2n13l1}(n_1, n_2)=(n, n+3l), \quad n, l\in \mathbb{N},\quad (n,l)\neq (0,0).\end{equation}
(Note that $0$ is a regular point of \eqref{3ode} if $n=l=0$.) Consequently,
\begin{equation}\label{localex}\varsigma_1=-n-l,\quad \varsigma_2=1-l, \quad \varsigma_3=n+2l+2,\end{equation}
\begin{equation}\label{alpha-nl}
\alpha=n^2+3nl+3l^2+3l+2n,
\end{equation}
\begin{equation}\label{beta-nl}
\beta=(l-1)(n+l)(n+2l+2)/2.
\end{equation}

\end{remark}

From the theory of linear ODEs with complex variable (see e.g. \cite[Chapter 1]{GP}), we collect some basic facts.
\begin{itemize}
\item (\ref{3ode}) always has a local solution $y_{a}(z)$ with $a=\varsigma_3=n+2l+2$.
\item Since the local exponents are all integers, (\ref{3ode}) may not have local solutions $y_{a}(z)$ of the form \eqref{ya} with $a\in \{-n-l, 1-l\}$ in general. If this happens, then (\ref{3ode}) must have solutions with logarithmic singularity at $0$, i.e. the local expansion of such solutions at $0$ is not of the form \eqref{ya} but contains $\ln z$ terms.
\item The singularity $0$ is called \emph{apparent} if any solution of (\ref{3ode}) has no logarithmic singularity at $0$, i.e. (\ref{3ode}) also has a local solution $y_{a}(z)$ with each $a\in\{-n-l, 1-l\}$. Clearly this implies that all solutions of (\ref{3ode}) are single-valued and meromorphic near the singularity $0$ and hence single-valued and meromorphic in $\mathbb{C}$ via the analytic continuation.
\item Suppose $0$ is an \emph{apparent} singularity of (\ref{3ode}). Then all solutions are meromorphic in $\mathbb{C}$, so the monodromy representation reduces to a group homeomorphism $\rho: \pi_1(E_{\tau})\to SL(3,\mathbb{C})$ as follows: Given a basis $(y_0,y_1,y_2)$ of (\ref{3ode}), for any $\ell\in \pi_1(E_{\tau})$, there is a monodromy matrix $\rho(\ell)\in SL(3,\mathbb{C})$ such that
     \[\ell^*(y_0,y_1,y_2)=(y_0,y_1,y_2)\rho(\ell),\]
     where $\ell^*y$ denotes the analytic continuation of $y(z)$ along the loop $\ell$. The fact $\rho(\ell)\in SL(3,\mathbb{C})$ (i.e. $\det \rho(\ell)=1$) follows from the simple fact that the Wronskian $\det W(y_0,y_1,y_2)\equiv C$ for some constant $C\neq 0$ independent of $z$, where
\begin{equation}\label{wrons}W(y_0,y_1,y_2):=\begin{pmatrix}y_0 & y_1 & y_2\\y_0' & y_1' & y_2'\\y_0'' & y_1'' & y_2''\end{pmatrix}.\end{equation}
     Now let $N_1, N_2\in SL(3,\mathbb{C})$ denote the monodromy matrices of (\ref{3ode}) with respect to the basic loops $z\to z+\omega_1$ and $z\to z+\omega_2$ in $\pi_1(E_{\tau})$ respectively. Then the monodromy group (i.e. the image of $\rho$) is generated by $N_1, N_2$. Since $\pi_1(E_{\tau})$ is abelian, we have
\begin{equation}\label{abelianm}N_1N_2=N_2N_1.\end{equation}
The monodromy matrices and so the monodromy group depend on the choice of the basis of (\ref{3ode}), and the monodromy groups with respect to different choices of basis are conjugated to each other.
\item The monodromy of (\ref{3ode}) is called \emph{unitary}, if there is a basis $(y_0,y_1,y_2)$ of (\ref{3ode}) such that the monodromy group with respect to $(y_0,y_1,y_2)$ is contained in the unitary group $SU(3)$, which is equivalent to $N_1, N_2\in SU(3)$.
\item It is well known that if (\ref{3ode}) has solutions with logarithmic singularity at $0$, then the local monodromy matrix at $0$ can not be a unitary matrix, so the monodromy can not be unitary.
\end{itemize}

One of our motivations comes from the classical theory (\cite{Whittaker-Watson}) for the monodromy of the Lam\'{e} equation \eqref{new-lm}, the local exponents of which at the singularity $0$ are distinct integers $-m, m+1$. Let $\widetilde{N}_j\in SL(2,\mathbb{C})$ denote the monodromy matrix of \eqref{new-lm} with respect to the basic loops $z\to z+\omega_j$ of $\pi_1(E_\tau)$.

\begin{Theorem}\cite{Whittaker-Watson, Gesztesy-Weikard, CLW}\label{thmA} \
\begin{itemize}
\item[(1)]
$0$ is always an apparent singularity of \eqref{new-lm} for any $B\in\mathbb{C}$, i.e. all solutions of \eqref{new-lm} are meromorphic in $\mathbb{C}$.

\item[(2)] There is a so-called Lam\'{e} polynomial $\ell_m(B)$ of degree $2m+1$ such that all the monodromy matrices of \eqref{new-lm} can be diagonized simultaneously if and only if $\ell_m(B)\neq 0$. In this case, up to a common conjugation,
\begin{equation}
\widetilde{N}_j=\operatorname{diag}(\lambda_{j}, \lambda_{j}^{-1}),\quad j=1,2,
\label{Mono-1}%
\end{equation}
with $\lambda_{j}=\lambda_j(B)\neq \pm 1$ for at least one $j\in \{1,2\}$.
\item[(3)]The monodromy is unitary if and only if $\ell_m(B)\neq 0$ and the corresponding $(\lambda_{1}, \lambda_{2})=(\lambda_{1}(B), \lambda_{2}(B))$ satisfies $|\lambda_{1}|=|\lambda_{2}|=1$.
\end{itemize}
\end{Theorem}

The Lam\'{e} polynomial $\ell_m(B)$ is also called the spectral polynomial in the theory of stationary KdV hierarchy; see e.g. \cite{Gesztesy-Weikard,Gesztesy-Weikard1,Takemura1,Takemura3}.
After Theorem \ref{thmA}, a natural question is whether there is $B$ such that the monodromy of (\ref{new-lm}) is unitary? It turns out this question is very subtle.
Define
\[\Omega_m:=\{\tau\in\mathbb{H}\,|\, \text{\it There is $B$ such that the monodromy of \eqref{new-lm} is unitary}\}.\]

\begin{Theorem}\label{thmB}\cite{Dahmen,CLW2,CL-AJM} $\Omega_m\neq \emptyset$ and $\Omega_m\neq \mathbb{H}$. In particular, $\Omega_m\cap i\mathbb{R}_{>0}=\emptyset$.
\end{Theorem}

Theorem \ref{thmB} says that there are $\tau$'s (i.e. those $\tau\in \Omega_m$) such that the monodromy of (\ref{new-lm}) is unitary for some $B$'s; but there are also $\tau$'s (such as $\tau\in i\mathbb{R}_{>0}$, i.e. $\tau$ is purely imaginary) such that the monodromy of (\ref{new-lm}) can not be unitary for any $B$.

In view of Theorems \ref{thmA}-\ref{thmB} for the second order Lam\'{e} equation, this paper aims to study the same question for the third order ODE \eqref{3ode}:

\medskip
\noindent{\bf (Q):} \emph{Does there exist $B\in\mathbb{C}$ such that $0$ is an apparent singularity and the monodromy is unitary for ODE \eqref{3ode}?}

\begin{remark}
This problem is fundamental but highly nontrivial in the theory of high order linear ODEs with complex variable. For example, consider the case that $n$ is odd and $l$ is even, then the three local exponents (see \eqref{localex}) are \emph{all odd}, and the standard Frobenius's method (see e.g. \cite[Chapter 1]{GP}) implies the existence of three polynomials $P_j(B)$ ($j=1,2,3$) of $B$ such that \emph{$0$ is an apparent singularity of \eqref{3ode} if and only if $B$ is a common root of the three polynomials}, i.e.
\[P_1(B)=P_2(B)=P_3(B)=0;\]
See Section 4 for details.
Since the explicit expressions of $P_j(B)$'s are unknown, it seems impossible to see why such $B$ exists. We need to develop new idea to overcome this difficulty.
\end{remark}

Our main results of this paper are as follows, which can be seen as the generalization of the above monodromy theory of the Lam\'{e} equation \eqref{new-lm} to the third order case. One will see that the third order case is much more complicated and some new phenomena appear.

\begin{theorem}\label{thm-ode-mo} Suppose $n\geq 1$ is odd. Then there is a polynomial $P_{n,l}(B)$ of degree $\frac{n+1}{2}$ such that $0$ is an apparent singularity of \eqref{3ode} if and only if $P_{n,l}(B)=0$. In this case,
\begin{itemize}
\item[(1)] If $l$ is odd, then the monodromy of \eqref{3ode} can not be unitary.
\item[(2)] If $l$ is even, then the monodromy group of \eqref{3ode} is the Klein-four group $K_4$ and so the monodromy is unitary.
\end{itemize}
Finally, the polynomial $P_{n,l}(B)$ has $\frac{n+1}{2}$ distinct roots expect for $\lceil \frac{n^2-1}{24}\rceil$ numbers of $\tau$'s modulo $SL(2,\mathbb{Z})$.
\end{theorem}

Given $a\in\mathbb{R}$, we denote
\[\lceil a\rceil:=\min\{b\in\mathbb{Z}\,|\, b\geq a\},\]
i.e. $\lceil a\rceil=a$ if $a$ is an integer and $\lceil a\rceil-1<a<\lceil a\rceil$ if $a$ is not an integer.

\begin{theorem}\label{thm-ode-mono} Suppose $n\geq 0$ is even. Then $0$ is always an apparent singularity of \eqref{3ode} for any $B\in\mathbb{C}$. Moreover,
\begin{itemize}
\item[(1)]
There exists a polynomial $Q_{n,l}(B)$ of odd degree
\[\deg Q_{n,l}(B)=\begin{cases}2n+3l+2 \quad\text{if $l$ is odd}\\
n+3l+1\quad\text{if $l$ is even}
\end{cases}\]
such that the monodromy matrices can be diagonalized simultaneously if and only if $Q_{n,l}(B)\neq 0$. In this case, up to a common conjugation,
\[N_j=\text{diag}(1, \lambda_{j},\lambda_{j}^{-1}), \quad \quad j=1,2,\]
    with $\lambda_{j}=\lambda_j(B)\neq \pm 1$ for at least one $j\in \{1,2\}$.
\item[(2)]The monodromy is unitary if and only if $Q_{n,l}(B)\neq 0$ and the corresponding $(\lambda_{1}, \lambda_{2})=(\lambda_{1}(B), \lambda_{2}(B))$ satisfies $|\lambda_{1}|=|\lambda_{2}|=1$.
    \end{itemize}

\end{theorem}

\begin{remark}
Theorem \ref{thm-ode-mo} gives a complete answer of (Q) for the case that $n$ is odd.
Define
\[\Omega_{n,l}:=\{\tau\in\mathbb{H}\,|\, \text{\it There is $B$ such that the monodromy of \eqref{3ode} is unitary}\}.\]
Then Theorem \ref{thm-ode-mo} shows that
\[\Omega_{n,l}=\begin{cases}\emptyset &\text{if $n, l$ are both odd},\\
\mathbb{H}&\text{if $n$ is odd and $l$ is even}.\end{cases}\]
This is a new phenomenon comparing to Theorems \ref{thmA}-\ref{thmB} for the Lam\'{e} equation.

For the case that $n$ is even, Theorem \ref{thm-ode-mono} behaves as an analogue of Theorem \ref{thmA}, and the polynomial $Q_{n,l}(B)$ plays a similar role as the Lam\'{e} polynomial $\ell_m(B)$ for the Lam\'{e} equation \eqref{new-lm}.
Note that Theorem \ref{thm-ode-mono} does not solve whether there exists $B$ such that the monodromy of \eqref{3ode} is unitary. This problem seems very difficult and we can not give a unified answer so far. However, we have the following interesting connection with the Lam\'{e} equation \eqref{new-lm} for the special case $l=1$.
\end{remark}
\begin{theorem}\label{thml=1}
Suppose $n\geq 0$ is even and $l=1$. Then $Q_{n,1}(B)=\ell_{n+2}(B)$ and $\Omega_{n,1}=\Omega_{n+2}$. Consequently, $\Omega_{n,1}\neq \emptyset$, $\Omega_{n,1}\neq \mathbb{H}$ and $\Omega_{n,1}\cap i\mathbb{R}_{>0}=\emptyset$.
\end{theorem}

Inspired by Theorem \ref{thml=1},
we conjecture that the assertions $\Omega_{n,l}\neq \emptyset$ and $\Omega_{n,l}\neq \mathbb{H}$  should hold for the general case that $n$ is even and $l$ is arbitrary. We should study this problem in future.

Our another motivation of studying \eqref{3ode} comes from its applications to the following $SU(3)$ Toda system on the torus $E_{\tau}$
\begin{equation}\label{Toda}
\begin{cases}
\Delta u+2e^{u}-e^v=4\pi n\delta_{0}\\
\Delta v+2e^{v}-e^u=4\pi (n+3l)\delta_{0}
\end{cases}
\;\text{ on
}\; E_{\tau},
\end{equation}
where $\Delta$ is the Laplace operator and $\delta_{0}$ is the Dirac measure at $0$. The Toda system is an important integrable system in mathematical physics \cite{BFO,g,Yang2001}; see also \cite{CL-noncritical,LNW,LWY} and references therein for
the recent development of the Toda system. In another paper \cite{CL-Toda}, we will apply Theorems \ref{thm-ode-mo}-\ref{thml=1} to prove the following result.

\begin{Theorem}\cite{CL-Toda} The Toda system \eqref{Toda} has even solutions (i.e. $u(z)=u(-z), v(z)=v(-z)$) if and only if there exists $B$ such that the monodromy of \eqref{3ode} is unitary. Furthermore,
\begin{itemize}
\item[(1)] If $n, l$ are both odd, then \eqref{Toda} has no even solutions.
\item[(2)] If $n$ is odd and $l$ is even, then \eqref{Toda} has infinitely many even solutions.
\item[(3)] If $n$ is even and $l=1$, then \eqref{Toda} has no even solutions provided that $\tau\in i\mathbb{R}_{>0}$, i.e. $E_{\tau}$ is a rectangular torus.
\end{itemize}

\end{Theorem}

Remark that since the methods of studying the second order Lam\'{e} equation (\ref{new-lm}) can not work for the third order ODE (\ref{3ode}), our proofs of Theorems \ref{thm-ode-mo}-\ref{thm-ode-mono} are not trivial at all. In fact, as one will see in later sections, we need to develop different approaches to treat the three cases (1) $n, l$ are both odd, (2) $n$ is odd and $l$ is even, and (3) $n$ is even, separately.
We believe that these new ideas have potential applications to some other high order ODEs.

The rest of this paper is organized as follows. In Section \ref{sec-dual}, we recall the notion of dual equations and study the existence of even elliptic solutions for this ODE \eqref{3ode}. In Section \ref{sec-unitary}, we study the case that both $n$ and $l$ are odd, and prove Theorem \ref{thm-ode-mo}-(1). In Section \ref{sec-Klein}, we study the case that $n$ is odd and $l$ is even. We prove that if $0$ is an apparent singularity, then the monodromy group is the Klein four-group. Theorem \ref{thm-ode-mo} follows from Sections 3-4. In Sections \ref{sec-monoth}, we study the case that $n$ is even, and prove Theorem \ref{thm-ode-mono}. Finally, Theorem \ref{thml=1} will be proved in Section \ref{sec-neven}.

\section{Preliminaries}

\label{sec-dual}

In this section, we recall the notion of dual equations and study the existence of even elliptic solutions for this ODE. These preliminary results will be applied in later sections.

\subsection{Dual equation}

Consider the following third order linear ODE
\begin{equation}\label{thirdode}
y'''+p_1(z)y'+p_0(z)y=0.
\end{equation}
For any two solutions $y_1, y_2$ of \eqref{thirdode}, we define the Wronskian
\begin{equation}\label{two-Wron}W(y_1,y_2):=y_1'y_2-y_1y_2'.\end{equation}
Then a direct computation shows that $W(y_1, y_2)$ solves the following equation
\begin{equation}\label{thirdode-dual}
Y'''+p_1(z)Y'+(p_1'(z)-p_0(z))Y=0.
\end{equation}
Equation (\ref{thirdode-dual}) is called the \emph{dual equation} of (\ref{thirdode}) in the literature. By direct computations it is easy to prove that
\begin{equation} \label{eq: data}
  \parbox{\dimexpr\linewidth-5em}{\it If $y_0, y_1, y_2$ are linearly independent solutions of \eqref{thirdode}, then $W(y_1,y_2)$, $W(y_2,y_0)$, $W(y_0, y_1)$ are linearly independent solutions of the dual equation \eqref{thirdode-dual}. In particular, if all solutions of \eqref{thirdode} are meromorphic, then so are all solutions of \eqref{thirdode-dual}, and vise versa.}
\end{equation}

Now the dual equation of the ODE \eqref{3ode} reads as
\begin{equation}\label{3ode-dual}Y'''-(\alpha\wp(z)+B)Y'-(\alpha+\beta)\wp'(z)Y=0.\end{equation}
By \eqref{alpha-nl}-\eqref{beta-nl}, it is easy to see that
\begin{equation}\label{beta-nl-dual}
-(\alpha+\beta)=-\frac{(l+1)(n+2l)(n+l+2)}{2},
\end{equation}
and the local exponents of the dual equation (\ref{3ode-dual}) at $0$ are
\begin{equation}\label{eq-s-7}
\tilde{\varsigma}_1=-n-2l,\quad \tilde{\varsigma}_2=l+1,\quad \tilde{\varsigma}_3=n+l+2,
\end{equation}
which are also three distinct integers.

\begin{remark}\label{rmk-dual} It is easy to see that the monodromy of \eqref{3ode} is unitary if and only if the monodromy of the dual equation \eqref{3ode-dual} is unitary.
Note that
\[\tilde{n}_1:=\tilde{\varsigma}_2-\tilde{\varsigma}_1-1=n+3l=n_2,\]
\[\tilde{n}_2:=\tilde{\varsigma}_3-\tilde{\varsigma}_2-1=n=n_1.\]
Thus, by replacing \eqref{3ode} with its dual equation \eqref{3ode-dual} if necessary, we only need to consider the case $n_1\leq n_2$ in this paper.
\end{remark}

\subsection{Even elliptic solutions}

Recalling that $\wp(z)=\wp(z;\tau)$ is the Weierstrass $\wp$-function with periods $\omega_1=1$ and $\omega_2=\tau$, we also denote $\omega_3=1+\tau$ and set $e_k=e_k(\tau)=\wp(\frac{\omega_k}{2};\tau)$ for $k=1,2,3$. It is well-known that
\begin{equation}\label{eq-wp}\wp'(z)^2=4\prod_{k=1}^3(\wp(z)-e_k)=4\wp(z)^3-g_2\wp(z)-g_3,\end{equation}
where $g_2=g_2(\tau), g_3=g_3(\tau)$ are known as Weierstrass invariants. They are modular forms of weights $4, 6$ respectively with respect to $SL(2,\mathbb{Z})$.

Since the potentials of (\ref{3ode}) are both elliptic functions, a natural question that plays an important role in studying the monodromy of (\ref{3ode}), is \emph{whether (\ref{3ode}) has elliptic solutions (i.e. solutions that are elliptic functions) or not}. In this section we prove that (\ref{3ode}) has actually an even elliptic solution provided that one of the negative local exponents at $0$ is even.

\begin{theorem} \label{even-elliptic-y1} Let $l=2k+1\geq 1$ be odd. Then \eqref{3ode} has a unique even elliptic solution of the form
\[y_0(z)=\sum_{j=0}^{k}C_j(B)\wp(z)^j,\]
where $C_k(B)=1$, $C_j(B)\in \mathbb{Q}[g_2,g_3][B]$ with degree $\deg C_j(B)=k-j$. Furthermore, $C_j(B)$ is homogenous of weight $k-j$, where the weights of $B, g_2, g_3$ are $1, 2, 3$ respectively.
\end{theorem}

\begin{proof} If $l=1$, then (\ref{beta-nl}) gives $\beta=0$ and so $y_0(z)=1$ is a solution. Therefore, it suffices to consider $l=2k+1$ with $k\geq 1$. Suppose that
\[y_0(z)=\sum_{j=0}^{k}C_jx^j,\quad C_k=1\]
is a solution of (\ref{3ode}), where we write $x=\wp(z)$ for convenience. Then $x':=\frac{dx}{dz}=\wp'(z)$
and
\[y_0'(z)=\wp'\sum_{j=0}^{k}jC_jx^{j-1},\]
\begin{align*}
y_0'''(z)=&(\wp')^3\sum_{j=0}^{k}j(j-1)(j-2)C_jx^{j-3}
+3\wp'\wp''\sum_{j=0}^{k}j(j-1)C_jx^{j-2}\\
&+\wp'''\sum_{j=0}^{k}jC_jx^{j-1}.
\end{align*}
By \[(\wp')^2=4\wp^3-g_2\wp-g_3=4x^3-g_2x-g_3,\]
and \[\wp''=6\wp^2-\tfrac{g_2}{2}=6x^2-\tfrac{g_2}{2},\quad \wp'''=12\wp'\wp=12\wp'x,\]
we obtain
\begin{align*}
\frac{y_0'''(z)}{\wp'}=&(4x^3-g_2x-g_3)\sum_{j=0}^{k}j(j-1)(j-2)C_jx^{j-3}
\\
&+3(6x^2-\tfrac{g_2}{2})\sum_{j=0}^{k}j(j-1)C_jx^{j-2}+12\sum_{j=0}^{k}jC_jx^{j}.
\end{align*}
Inserting these into (\ref{3ode}) easily leads to
\begin{align}\label{recursive-y1}
\sum_{j=0}^k\Big(&\varphi_j C_j-(j+1)BC_{j+1}-(j+1)(j+\tfrac{3}{2})(j+2)g_2 C_{j+2}\\
&-(j+1)(j+2)(j+3)g_3C_{j+3}\Big)x^j=0,\nonumber
\end{align}
where $C_{k+1}=C_{k+2}=C_{k+3}:=0$ and
\begin{align}\label{varphi}\varphi_j:= & 4j(j-1)(j-2)+18j(j-1)+12j-\alpha j+\beta\\
=&\tfrac{1}{2}(2j-l+1)(2j-n-l)(2j+n+2l+2).\nonumber
\end{align}
Here we used (\ref{alpha-nl})-(\ref{beta-nl}) to obtain the second equality of (\ref{varphi}).
Therefore, $y_0(z)$ is a solution of (\ref{3ode}) if and only if $C_j$ satisfies the following recursive relation
\begin{align}\label{recursive-y1-1}\varphi_j C_j=&(j+1)BC_{j+1}+(j+1)(j+\tfrac{3}{2})(j+2)g_2 C_{j+2}\\
&+(j+1)(j+2)(j+3)g_3C_{j+3},\quad j=k, k-1, \cdots, 0.\nonumber
\end{align}
Since (\ref{varphi}) and $l=2k+1$ implies $\varphi_k=0$ and $\varphi_j\neq 0$ for $j\in [0, k-1]$, we see that (\ref{recursive-y1-1}) for $j=k$ holds automatically, and $C_j$, $j=k-1, k-2, \cdots, 0$, can be uniquely solved in terms of $C_{j+1}, C_{j+2}, C_{j+3}$ via (\ref{recursive-y1-1}) by induction. In particular, the resulting $C_j=C_j(B)\in\mathbb{Q}[g_2,g_3][B]$ satisfies the desired properties. The proof is complete.
\end{proof}

The same argument as Theorem \ref{even-elliptic-y1} yields the following results; we omit the proofs here.

\begin{theorem} \label{even-elliptic-y11} Let $n, l\geq 0$ be even, $n+l\geq 2$ and denote $n+l=2k$. Then \eqref{3ode} has a unique even elliptic solution of the form
\[y_0(z)=\sum_{j=0}^{k}C_j(B)\wp(z)^j,\]
where $C_{k}(B)=1$, $C_j(B)\in \mathbb{Q}[g_2,g_3][B]$ with degree $\deg C_j(B)=k-j$. Furthermore, $C_j(B)$ is homogenous of weight $k-j$, where the weights of $B, g_2, g_3$ are $1, 2, 3$ respectively.
\end{theorem}

\begin{theorem}\label{even-elliptic-Y1} Let $n\geq 0$ be even and denote $n+2l=2m_0$. Then the dual equation \eqref{3ode-dual} has a unique even elliptic solution of the form
\[Y_0(z)=\sum_{j=0}^{m_0}D_j(B)\wp(z)^j,\]
where $D_{m_0}(B)=1$, $D_j(B)\in \mathbb{Q}[g_2,g_3][B]$ with degree $\deg D_j(B)=m_0-j$. Furthermore, $D_j(B)$ is homogenous of weight $m_0-j$, where the weights of $B, g_2, g_3$ are $1, 2, 3$ respectively.
\end{theorem}

\section{The case $n$, $l$ odd: Monodromy can not be unitary}

\label{sec-unitary}

In this section, we consider the case that $n, l\geq 1$ are both odd and prove Theorem \ref{thm-ode-mo}-(1).
First, we need to study the condition on $B$ such that $0$ is an apparent singularity. For this purpose, we need the following remark.

\begin{remark}\label{rrr}
Note that ODE (\ref{3ode}) is invariant under $z\to -z$, namely if $y(z)$ is a solution, then so is $y(-z)$. Thus, if
\[y(z)=z^{a}\sum_{j=0}^\infty c_j z^j,\quad c_0\neq 0\]
is a local solution of (\ref{3ode}) near $0$, then so is
\[y(-z)=(-z)^{a}\sum_{j=0}^\infty c_j (-z)^j.\]
and so we obtain a local solution of the following form (note $(-1)^{2a}=1$ because $a$ is a local exponent and hence an integer)
\begin{equation}\label{eq-loc}\frac{y(z)+(-1)^a y(-z)}{2}=z^{a}\sum_{j=0}^\infty \tilde{c}_j z^{2j},\quad\text{where }\tilde{c}_j=c_{2j}.\end{equation}
Therefore, to prove $0$ is apparent is equivalent to prove the existence of local solutions of the form (\ref{eq-loc}) for $a\in \{-n-l, -l+1\}$.
\end{remark}

Recalling (\ref{localex}) that among the local exponents, both $-n-l$ and $-l+1$ are even, so the local solution with exponent $-n-l$ might be logarithmic at $0$ in general.

\begin{lemma}\label{lemma-app} Suppose $n, l\geq 1$ are both odd and write $n+l=2k$. Then there exists a monic polynomial $P_{n,l}(B)\in \mathbb{Q}[g_2, g_3][B]$ with degree $k-\frac{l-1}{2}=\frac{n+1}{2}$ and homogenous weight $\frac{n+1}{2}$ such that $0$ is an apparent singularity of \eqref{3ode} if and only if $P_{n,l}(B)=0$. Here the weights of $B, g_2, g_3$ are $1, 2, 3$ respectively.
\end{lemma}

\begin{proof}
Thanks to Theorem \ref{even-elliptic-y1} which says that (\ref{3ode}) has an even elliptic solution with the local exponent $-l+1$ at $0$, the logarithmic singularity can only appear at the local solution with exponent $-n-l$.
Therefore, $0$ is an apparent singularity if and only if (\ref{3ode}) has a local solution with exponent $-n-l=-2k$ of the form
\begin{equation}\label{loc-sol}y(z)=z^{-n-l}\sum_{j=0}^{\infty}c_j z^{2j}=\sum_{j=0}^{\infty}c_j z^{2j-2k},\quad c_0=1.\end{equation}

The following proof is standard by applying Frobenius's method.
Applying the well-known Laurent expansion of $\wp(z)$, we have
\begin{equation}\label{eq-wp-ex}\wp(z)+\tfrac{B}{\alpha}=\sum_{j=0}^{\infty}B_j z^{2j-2}, \quad\text{where}\; B_0=1,\, B_1=\tfrac{B}{\alpha},\, B_2=\tfrac{g_2}{20}\; \text{etc.,}\end{equation}
and $B_j\in\mathbb{Q}[g_2, g_3]$ with homogeneous weight $j$ for any $j\geq 2$. Then
\[\wp'(z)=\sum_{j=0}^{\infty}(2j-2)B_j z^{2j-3}.\]
Inserting these into (\ref{3ode}) and a direct computation leads to
\begin{align}\label{eq-appapp}&\sum_{j=0}^{\infty}\Big((2j-2k)(2j-2k-1)(2j-2k-2)c_j\\
&-\alpha\sum_{i=0}^{j}(2j-2i-2k)B_ic_{j-i}
+\beta\sum_{i=0}^{j}(2i-2)B_ic_{j-i}\Big)z^{2j-2k-3}=0.\nonumber\end{align}
Therefore, $y(z)$ is a solution if and only if
\begin{align}\label{rec-app}
\phi_j c_j=&(2j-2-2k)Bc_{j-1}\\&
+\sum_{i=2}^{j}[\alpha(2j-2i-2k)-\beta(2i-2)]B_ic_{j-i},\quad\forall j\geq 0,\nonumber
\end{align}
where $c_{-2}=c_{-1}:=0$ and
\begin{align*}
\phi_j:=&(2j-2k)(2j-2k-1)(2j-2k-2)-\alpha(2j-2k)-2\beta\\
=&8j(j-k+\tfrac{l-1}{2})(j-n-1-\tfrac{3l}{2}).
\end{align*}
Note that $\phi_j=0$ if and only if $j\in\{0, k-\frac{l-1}{2}\}$, and (\ref{rec-app}) with $j=0$ holds automatically.
By (\ref{rec-app}) and the induction argument, for $1\leq j\leq k-\frac{l-1}{2}-1$, $c_j$ can be uniquely solved as $c_j=c_j(B)\in \mathbb{Q}[g_2,g_3][B]$ with degree $j$ in $B$ and homogenous weight $j$. Consequently, the RHS of (\ref{rec-app}) with $j=k-\frac{l-1}{2}$ is a polynomial in $\mathbb{Q}[g_2,g_3][B]$ with degree $k-\frac{l-1}{2}$ and homogenous weight $k-\frac{l-1}{2}$. Define $P_{n,l}(B)$ to be the corresponding monic polynomial. Then the standard Frobenius theory shows that  (\ref{3ode}) has a local solution $y(z)$ of the form (\ref{loc-sol}) if and only if  $P_{n,l}(B)=0$. This completes the proof.
\end{proof}

After Lemma \ref{lemma-app}, the remaining problem is whether the monodromy is unitary or not for those $B$ satisfying $P_{n,l}(B)=0$, which is challenging in general. To overcome this difficulty, our new idea is to prove that ODE (\ref{3ode}) has two linearly independent even elliptic solutions, one of which has been obtained in Theorem \ref{even-elliptic-y1}. The second one is established in the following lemma.

\begin{lemma}\label{lemma-twoelliptic} Suppose $n, l\geq 1$ are both odd and write $n+l=2k$. Then there exists a polynomial $\widetilde{P}_{n,l}(B)\in \mathbb{Q}[g_2, g_3][B]$ of degree $k-\frac{l-1}{2}=\frac{n+1}{2}$ and homogenous weight $\frac{n+1}{2}$ such that \eqref{3ode} has an even elliptic solution of degree $k$ in $\wp(z)$ if and only if $\tilde{P}_{n,l}(B)=0$.
\end{lemma}

\begin{proof}
The proof is similar to that of Theorem \ref{even-elliptic-y1}. Suppose that
\begin{equation}\label{sec-elliptic}y(z)=\sum_{j=0}^{k}C_jx^j,\quad x=\wp(z),\; C_k=1\end{equation}
is a solution of (\ref{3ode}). Then this is equivalent to that $C_j$ satisfies the recursive relations (\ref{varphi})-(\ref{recursive-y1-1}). Note from $n+l=2k$ that $\varphi_j=0$ if and only if $j\in \{k,\frac{l-1}{2}\}$. By (\ref{recursive-y1-1}), for $\frac{l-1}{2}+1\le j\le k-1$, $C_j$ can be uniquely solved as $C_j=C_j(B)\in\mathbb{Q}[g_2,g_3][B]$ with degree $k-j$ and homogenous weight of $k-j$. Then the RHS of (\ref{recursive-y1-1}) with $j=\frac{l-1}{2}$ is a polynomial in $\mathbb{Q}[g_2,g_3][B]$ with degree $k-\frac{l-1}{2}$ and homogenous weight of $k-\frac{l-1}{2}$. Denote this corresponding monic polynomial by $\widetilde{P}_{n,l}(B)$. Then $y(z)$ being a solution implies $\widetilde{P}_{n,l}(B)=0$.

Conversely, if $\widetilde{P}_{n,l}(B)=0$, then (\ref{recursive-y1-1}) with $j=\frac{l-1}{2}$ holds automatically and $C_{\frac{l-1}{2}}$ can be arbitrary. Take any value for $C_{\frac{l-1}{2}}$, then $C_j$, $0\leq j\leq \frac{l-1}{2}-1$ can be uniquely determined via (\ref{recursive-y1-1}), namely (\ref{3ode}) has a solution $y(z)$ of the form (\ref{sec-elliptic}). Remark that the freedom of $C_{\frac{l-1}{2}}$ follows from the other even elliptic solution proved in Theorem \ref{even-elliptic-y1}. The proof is complete.
\end{proof}

The next question is whether the above two polynomials $P_{n,l}(B)$ and $\widetilde{P}_{n,l}(B)$ coincide or not, which plays a crucial role in the proof of the monodromy being not unitary. To settle this question, we need the follow lemma.

\begin{lemma}\label{lemma-real} The polynomial $\widetilde{P}_{n,l}(B)$ in Lemma \ref{lemma-twoelliptic} has $k-\frac{l-1}{2}=\frac{n+1}{2}$ distinct roots expect for $\lceil \frac{n^2-1}{24}\rceil$ numbers of $\tau$'s modulo $SL(2,\mathbb{Z})$.
\end{lemma}

\begin{proof}
The key idea of this proof is to consider the special case $\tau=i=\sqrt{-1}$ and prove that $\widetilde{P}_{n,l}(B)$ has real distinct roots. For $\tau=i$, it is well known that $g_3=0$ and $g_2>0$, so the recursive formula (\ref{recursive-y1-1}) becomes
\begin{align}\label{recursive-y1-10}\varphi_j C_j=(j+1)BC_{j+1}+(j+1)(j+\tfrac{3}{2})(j+2)g_2 C_{j+2},\quad 0\leq j\leq k,
\end{align}
with $k=\frac{n+l}{2}$ and
\begin{equation}\label{varphi-0}\varphi_j=\tfrac{1}{2}(2j-l+1)(2j-n-l)(2j+n+2l+2).
\end{equation}
Comparing to the original recursive formula (\ref{recursive-y1-1}) for general $\tau$, the advantage of (\ref{recursive-y1-10}) is that it contains only three terms $C_j, C_{j+1}, C_{j+2}$, from which we can apply the method of Sturm sequence to prove that $\widetilde{P}_{n,l}(B)$ has real distinct roots.

Recall that $C_k=1$, $C_{k+1}=C_{k+2}=0$, and $C_j=C_j(B)$ are polynomials in $\mathbb{Q}[g_2][B]\subset \mathbb{R}[B]$ with degree $k-j$ for $\frac{l-1}{2}+1\leq j\leq k-1$. Write $k_0=\frac{l-1}{2}$ for convenience.

{\bf Step 1.} We prove that for $k_0+1\leq j\leq k-1$, $C_j(B)$ has real distinct roots, denoted by $r_1^j<\cdots<r_{k-j}^j$, such that
\begin{equation}\label{root1}
r_1^j<r_1^{j+1}<r_2^j<\cdots<r_{k-j-1}^{j}<r_{k-j-1}^{j+1}<r_{k-j}^j.
\end{equation}
The proof is by the method of Sturm sequence; we give the details here for completeness.

Recalling (\ref{varphi-0}) that $\varphi_j<0$ for $k_0+1\leq j\leq k-1$, we easily obtain the following properties for all $k_0+1\leq j\leq k-1$ from (\ref{recursive-y1-10}):

{\bf (P1)} Up to a positive constant, the leading term in $C_{j}(B)$ is $(-1)^{k-j}B^{k-j}$.

{\bf (P2)} If $C_{j+1}(B)=0$ and $C_{j+2} (B)\neq 0$ for $B\in \mathbb{R}$, then
$C_{j} (B)C_{j+2} (B)<0$.

Now we prove (\ref{root1}) by induction. The case $j=k-1$ is trivial because $\deg C_{k-1}(B)=1$. Let $k_0<m<k-1$ and
assume that the statement is true for any $j>m$. We prove it for $j=m$. From the assumption of the
induction,
\begin{equation}\label{s-k}
r_{1}^{m+1}<r_{1}^{m+2}<r_{2}^{m+1}<\cdots
<r_{k-m-2}^{m+1}<r_{k-m-2}^{m+2}<r_{k-m-1}^{m+1}.
\end{equation}
Recall {\bf (P1)} that
\begin{equation}\label{ck-1}
\lim_{B\rightarrow-\infty} C_{m+2}(B) =+\infty, \; \lim_{B \rightarrow
+\infty} C_{m+2}(B) = (-1)^{k-m-2}\infty.
\end{equation}
Since $r_{j}^{m+2}$, $1\le j\le k-m-2$, are all the roots of  $C_{m+2}(B)$, it follows from (\ref{s-k}) and (\ref{ck-1}) that
\begin{equation}\label{c-k-1-k}
C_{m+2}(r_j^{m+1})\sim (-1)^{j-1},\quad \forall j\in [1,k-m-1].
\end{equation}
Here $c\sim (-1)^j$ means $c=(-1)^j\tilde{c}$ for some $\tilde{c}>0$. Then we see from {\bf (P2)} that
\[C_{m}(r_j^{m+1})\sim (-1)^{j},\quad \forall j\in [1,k-m-1].\]
On the other hand, {\bf (P1)} implies
\[\lim_{B \rightarrow-\infty} C_{m}(B) =+\infty, \; \lim_{B \rightarrow
+\infty} C_{m}(B) = (-1)^{k-m} \infty.\]
From here,
it follows from the intermediate value theorem that the polynomial
$C_{m}(B)$ has $k-m$ real distinct roots $r_{j}^{m} $ $(1\leq j\leq
k-m)$ such that \begin{equation}\label{sk1}r_{1}^{m}<r_{1}^{m+1}<r_{2}^{m}
<\dots<r_{k-m-1}^{m}<r_{k-m-1}^{m+1}<r_{k-m}^{m}.\end{equation}
This proves (\ref{root1}) for all $k_0+1\leq j\leq k-1$.

{\bf Step 2.} We prove that $\widetilde{P}_{n,l}(B)$ has $k-k_0$ real distinct roots.

Define
\begin{equation}\label{Toda-eq-3-1}L(B):=(k_0+1)BC_{k_0+1}+(k_0+1)(k_0+\tfrac{3}{2})(k_0+2)g_2 C_{k_0+2}.\end{equation}
Then the definition of $\widetilde{P}_{n,l}(B)$ in Lemma \ref{lemma-twoelliptic} implies $L(B)=c\widetilde{P}_{n,l}(B)$, where $c\neq 0$ is the coefficient of the leading term of $L(B)$. So it suffices to prove that $L(B)$ real distinct roots.

Recall {\bf (P1)} that
\begin{equation}\label{ck-1}
\lim_{B\rightarrow-\infty} C_{k_0+2}(B) =+\infty, \; \lim_{B \rightarrow
+\infty} C_{k_0+2}(B) = (-1)^{k-k_0-2}\infty.
\end{equation}
By (\ref{root1}) with $j=k_0+1$, we have (similar to (\ref{c-k-1-k}))
\[
C_{k_0+2}(r_i^{k_0+1})\sim (-1)^{i-1},\quad \forall i\in [1,k-k_0-1],
\]
and so it follows from (\ref{Toda-eq-3-1}) and $C_{k_0+1}(r_i^{k_0+1})=0$ that
\[L(r_i^{k_0+1})\sim C_{k_0+2}(r_i^{k_0+1})\sim (-1)^{i-1}, \quad \forall i\in [1,k-k_0-1].\]
On the other hand, up to a positive constant, the leading term of $L(B)$ is $(-1)^{k-k_0-1}B^{k-k_0}$, the same as that of $BC_{k_0+1}$, which implies
\[\lim_{B \rightarrow-\infty} L(B) =-\infty, \; \lim_{B \rightarrow
+\infty} L(B) = (-1)^{k-k_0-1} \infty.\]
Therefore,
$L(B)$ has $k-k_0$ real distinct roots $r_{i}^{k_0} $ $(1\leq i\leq
k-k_0)$ such that \[r_{1}^{k_0}<r_{1}^{k_0+1}<r_{2}^{k_0}
<\dots<r_{k-k_0-1}^{k_0}<r_{k-k_0-1}^{k_0+1}<r_{k-k_0}^{k_0}.\]

{\bf Step 3.} We complete the proof. Recall that $g_2(\tau), g_3(\tau)$ are modular forms of weights $4, 6$ respectively, with respect to $SL(2,\mathbb{Z})$. Since $\widetilde{P}_{n,l}(B)\in \mathbb{Q}[g_2(\tau), g_3(\tau)][B]$ is of homogenous weight $\frac{n+1}{2}$ with the weights of $B, g_2, g_3$ being $1, 2, 3$ respectively, then
the discriminant $D(\tau)$ of $\widetilde{P}_{n,l}(B)$ is a modular form of weight $\frac{n^2-1}{2}$ with respect to $SL(2,\mathbb{Z})$. By Step 2 we have $D(i)\neq 0$ and so $D(\tau)\neq 0$ for any $\tau\equiv i$ modulo $SL(2,\mathbb{Z})$, i.e. $D(\tau)\not\equiv 0$. Since \[\frac{n^2-1}{2}\equiv 0 \text{ or } 4 \quad\mod 12\quad\text{for any odd $n$},\]
it follows from the well-known theorem of counting zeros of modular forms (see e.g. \cite{Serre}) that $D(\tau)$ has at most
$\lceil \frac{n^2-1}{24}\rceil$ zeros
modulo $SL(2,\mathbb{Z})$, so $D(\tau)\neq0$ except for $\lceil \frac{n^2-1}{24}\rceil$ numbers of $\tau$'s modulo $SL(2,\mathbb{Z})$. This proves that $\widetilde{P}_{n,l}(B)$ has $\frac{n+1}{2}$ distinct roots expect for $\lceil \frac{n^2-1}{24}\rceil$ numbers of $\tau$'s modulo $SL(2,\mathbb{Z})$.
\end{proof}

\begin{lemma}\label{lemma-poly} Recall the monic polynomials $P_{n,l}(B)$ and $\widetilde{P}_{n,l}(B)$ in Lemmas \ref{lemma-app} and \ref{lemma-twoelliptic} respectively. Then $P_{n,l}(B)=\widetilde{P}_{n,l}(B)$. Consequently, $P_{n,l}(B)$ has $\frac{n+1}{2}$ distinct roots expect for $\lceil \frac{n^2-1}{24}\rceil$ numbers of $\tau$'s modulo $SL(2,\mathbb{Z})$.
\end{lemma}

\begin{proof} Fix any $\tau$ such that $\widetilde{P}_{n,l}(B)$ has distinct roots. Assume $\widetilde{P}_{n,l}(B)=0$. Then Lemma \ref{lemma-twoelliptic} says that (\ref{3ode}) has an even elliptic solution with local exponent $-n-l$ at $0$. Together with Theorem \ref{even-elliptic-y1} which says that (\ref{3ode}) has an even elliptic solution with local exponent $-l+1$ at $0$, we see that all solutions of (\ref{3ode}) are meromorphic at $0$, i.e. $0$ is an apparent singularity. Then Lemma \ref{lemma-app} implies $P_{n,l}(B)=0$, i.e. $\widetilde{P}_{n,l}(B)|P_{n,l}(B)$. Since $\deg \widetilde{P}_{n,l}(B)=\deg P_{n,l}(B)=\frac{n+1}{2}$ and $\widetilde{P}_{n,l}(B)$ has distinct roots, we conclude that $P_{n,l}(B)=\widetilde{P}_{n,l}(B)$. It follows from Lemma \ref{lemma-real} that $P_{n,l}(B)=\widetilde{P}_{n,l}(B)$ for almost all $\tau$ and so for all $\tau$ by continuity with respect to $\tau$.
\end{proof}

Now we are in  the position to prove the following result.

\begin{theorem}[=Theorem \ref{thm-ode-mo}-(1)]\label{thm-nonexistence}
Suppose (\ref{n2n13l1}) with $n$ odd and $l$ odd. Then the monodromy of (\ref{3ode}) can not be unitary for any $B\in\mathbb{C}$.
\end{theorem}

\begin{proof}
Suppose that for some $B$, the monodromy of (\ref{3ode}) is unitary, i.e. the monodromy group is conjugate to a subgroup of the unitary group $SU(3)$. Then $0$ is an apparent singularity, so the previous Lemmas \ref{lemma-app}-\ref{lemma-poly} show that (\ref{3ode}) has an even elliptic solution $y_1(z)$ with local exponent $-n-l$. Note from Theorem \ref{even-elliptic-y1} that (\ref{3ode}) has another even elliptic solution $y_0(z)$ with local exponent $-l+1$. Recalling the monodromy matrices $N_1, N_2\in SL(3,\mathbb{C})$ in (\ref{abelianm}), i.e. $\det N_j=1$, it follows that $\lambda=1$ is the only eigenvalue of both $N_1$ and $N_2$ with multiplicity $3$. Since $N_j$ is conjugate to a unitary matrix, so it can be diagonalized and in conclusion,
\[N_1=N_2=I_3:=\left(\begin{smallmatrix}1 & &\\
&1 & \\
& & 1\end{smallmatrix}\right).\]
Therefore, all solutions of (\ref{3ode}) are elliptic functions. In particular, let $y_2(z)$ be a local solution of (\ref{3ode}) with local exponent $n+2l+2>0$, then $y_2(z)$ is an elliptic function. However, $y_2(z)$ has no poles but has zeros at $0$, which implies $y_2(z)\equiv 0$, a contradiction. Thus, the monodromy of (\ref{3ode}) can not be unitary for any $B\in\mathbb{C}$. The proof is complete.
\end{proof}

\section{The case $n$ odd, $l$ even: Apparent imply Klein four-group}
\label{sec-Klein}

In this section, we consider the case that $n$ is odd and $l$ is even, and prove Theorem \ref{thm-ode-mo}-(2).
Note that the local exponents $-n-l$, $-l+1$ and $n+2l+2$ are all \emph{odd},
so the local solutions with exponents $-n-l$ or $-l+1$ might be logarithmic at $0$ in general.

First we study the condition on $B$ to guarantee that $0$ is an apparent singularity of (\ref{3ode}).
The first main result of this section is

\begin{theorem}
\label{thm-noddleven-apparent}
Suppose $n\geq 1$ is odd and  $l\geq 0$ is even. Then there is a monic polynomial $P_{n,l}(B)\in\mathbb{Q}[g_2,g_3][B]$ with degree $\frac{n+1}{2}$ in $B$ and also homogenous weight $\frac{n+1}{2}$ such that $0$ is an apparent singularity of \eqref{3ode} if and only if $P_{n,l}(B)=0$.
Here the weights of $B, g_2, g_3$ are $1, 2, 3$ respectively, and
\begin{equation}\label{remark-n1l2}
P_{1,l}(B)=B,\quad P_{3,l}(B)=B^2-3(l+2)^2g_2.
\end{equation}
\end{theorem}

\begin{remark} The conclusion of Theorem \ref{thm-noddleven-apparent} is the same as that of Lemma \ref{lemma-app}, which together imply the first assertion of Theorem \ref{thm-ode-mo}, but the proof is quite different. One obvious difference from the case of $n,l$ being both odd is that (\ref{3ode}) has no elliptic solutions for the case of $n$ odd and $l$ even (see Theorem \ref{thm-nodd-k4} for the proof).  We want to emphasize that in contrast with Lemma \ref{lemma-app}, Frobenius's method is not enough for us to prove Theorem \ref{thm-noddleven-apparent}, and the following proof of Theorem \ref{thm-noddleven-apparent} requires new ideas and is highly nontrivial. More precisely, the necessary part of Theorem \ref{thm-noddleven-apparent} can be proved by applying Frobenius's method, but the sufficient part can not, and we need to develop new ideas to settle this problem. The reason is that since the three local exponents are all odd, Frobenius's method will imply the existence of three polynomials of $B$ such that $0$ is apparent if and only if $B$ is a common root of these three polynomials, but the problem is that it seems too difficult to see whether these three polynomials have common roots or not. See Remark \ref{rrrmmm} below.
\end{remark}

First we need the following simple lemma.

\begin{lemma}\label{lemma-4n-1}
Suppose $n\geq 1$ is odd,  $l\geq 0$ is even and $0$ is an apparent singularity of \eqref{3ode} for some $B$. Then all solutions of \eqref{3ode} with this $B$ are meromorphic and odd.
\end{lemma}

\begin{proof}
Since $0$ is an apparent singularity, all solutions of (\ref{3ode}) are meromorphic in $\mathbb{C}$. If $y(z)\neq -y(-z)$ for some solution $y(z)$ of (\ref{3ode}), then $y(z)+y(-z)\neq 0$ is an even solution of (\ref{3ode}) and so its local exponent at $0$ is even, a contradiction with the fact that all local exponents of (\ref{3ode}) at $0$ are odd.
\end{proof}

Write the ODE (\ref{3ode}) as
\begin{equation}\label{3ode-L}
Ly=0,\quad \text{where}\quad L:=\frac{d^3}{dz^3}-(\alpha\wp(z)+B)\frac{d}{dz}+\beta\wp'(z).
\end{equation}
Now we apply Frobenius's method to prove the necessary part.

\begin{proof}[Proof of the necessary part of Theorem \ref{thm-noddleven-apparent}]
Suppose that $0$ is an apparent singularity.
Then as pointed out in Remark \ref{rrr}, (\ref{3ode}) has local solutions of the form
\begin{equation}\label{nloc-sol}y(z)=\sum_{j=0}^{\infty}c_j z^{2j-t}\quad\text{for }\;t\in \{n+l, l-1\},\quad \text{where }\;c_0\neq 0.\end{equation}
Here we only use the local solution with $t=n+l$ just as in Lemma \ref{lemma-app}.
Inserting (\ref{nloc-sol}) with $t=n+l$ and the Laurent expansion (\ref{eq-wp-ex}) of $\wp(z)$ into $Ly$, it follows from (\ref{eq-appapp})-(\ref{rec-app}) that
\begin{align}
\label{noddee}Ly=&\sum_{j=0}^{\infty}\Big(\phi_j c_j-(2j-2-n-l)Bc_{j-1}\\
&-\sum_{i=2}^{j}[\alpha(2j-2i-n-l)-\beta(2i-2)]B_ic_{j-i}\Big)z^{2j-n-l-3},\nonumber\end{align}
and so $Ly=0$ if and only if
\begin{align}\label{nrec-app}
\phi_j c_j=&(2j-2-n-l)Bc_{j-1}\\&
+\sum_{i=2}^{j}[\alpha(2j-2i-n-l)-\beta(2i-2)]B_ic_{j-i},\quad\forall j\geq 0,\nonumber
\end{align}
where $c_{-2}=c_{-1}:=0$ and
\begin{align*}
\phi_j:=&(2j-n-l)(2j-n-l-1)(2j-n-l-2)-\alpha(2j-n-l)-2\beta\\
=&8j(j-\tfrac{n+1}{2})(j-\tfrac{2n+3l+2}{2}).
\end{align*}
Note $2j-2-n-l\neq 0$ for any $j$ because $n+l$ is odd, and $\phi_j=0$ if and only if $j\in\{0, \frac{n+1}{2}, \frac{2n+3l+2}{2}\}$. This is the difference comparing to Lemma \ref{lemma-app}.

Again (\ref{nrec-app}) with $j=0$ holds automatically.
By (\ref{nrec-app}) and the induction argument, for $1\leq j\leq \frac{n+1}{2}-1$, $c_j$ can be uniquely solved as
\[\frac{c_j}{c_0}=c_j(B)\in \mathbb{Q}[g_2,g_3][B]\]
with degree $j$ in $B$ and homogenous weight $j$.
Consequently, the RHS of (\ref{nrec-app}) with $j=\tfrac{n+1}{2}$ equals $c_0$ multiplying a polynomial in $\mathbb{Q}[g_2,g_3][B]$ with degree $\frac{n+1}{2}$ and homogenous weight $\frac{n+1}{2}$. Define $P_{n,l}(B)$ to be the corresponding monic polynomial, then there is $r_0\in\mathbb{Q}\setminus\{0\}$ such that
\[\text{\it the RHS of (\ref{nrec-app}) with $j=\tfrac{n+1}{2}$ equals } c_0r_0P_{n,l}(B).\]
Thus $c_0\neq 0$ and (\ref{nrec-app}) with $j=\frac{n+1}{2}$ imply $P_{n,l}(B)=0$. Clearly (\ref{remark-n1l2}) can be computed directly in the above proof.

Since $\phi_j=0$ also at $j=\frac{2n+3l+2}{2}$, at the moment we can only conclude that $P_{n,l}(B)=0$ is a necessary but might not sufficient condition for (\ref{3ode}) having a solution of the form (\ref{nloc-sol}) with $t=n+l$.

Clearly the above argument shows that if (\ref{nrec-app}) holds for any $j\leq\frac{n+1}{2}-1$, then
\begin{align}\label{odd-appapp1}Ly=&-c_0r_0P_{n,l}(B)z^{-l-2}+\sum_{j=\frac{n+1}{2}+1}^{\infty}\Big(\phi_j c_j-(2j-2-n-l)Bc_{j-1}\\
&-\sum_{i=2}^{j}[\alpha(2j-2i-n-l)-\beta(2i-2)]B_ic_{j-i}\Big)z^{2j-n-l-3}\nonumber
\end{align}
holds for arbitrary choices of $c_0\neq 0$ and $c_{\frac{n+1}{2}}$. This fact will be used in the proof of the sufficient part.
\end{proof}

\begin{remark}\label{rrrmmm}
In the above proof, $\phi_j=0$ at $j=\frac{2n+3l+2}{2}$ will leads to another polynomial $\hat{P}_{n,l}(B)$ such that (\ref{3ode}) has a solution of the form (\ref{nloc-sol}) with $t=n+l$ if and only if $P_{n,l}(B)=\hat{P}_{n,l}(B)=0$. Furthermore, if we consider the local solution $y(t)$ in (\ref{nloc-sol}) with $t=l-1$, we will obtain a third polynomial $\tilde{P}_{n,l}(B)$ such that (\ref{3ode}) has a solution of the form (\ref{nloc-sol}) with $t=l-1$ if and only if $\tilde{P}_{n,l}(B)=0$. Therefore, Frobenius's method implies that $0$ is an apparent singularity if and only if \[P_{n,l}(B)=\hat{P}_{n,l}(B)=\tilde{P}_{n,l}(B)=0.\] This indicates that to promise the validity of the sufficient part of Theorem \ref{thm-noddleven-apparent}, we need to show
\[P_{n,l}(B)=0\quad\Longrightarrow\quad \hat{P}_{n,l}(B)=\tilde{P}_{n,l}(B)=0,\]
which seems too difficult to be proved via Frobenius's method.
\end{remark}

\begin{proof}[Proof of the sufficient part of Theorem \ref{thm-noddleven-apparent}] Suppose $P_{n,l}(B)=0$. Here we develop a new idea to prove that $0$ is an apparent singularity. We descend (\ref{3ode}) to $\mathbb{P}^1$ under the double cover $\wp: E_{\tau}\to \mathbb{P}^1$. Let $x=\wp(z)$, $\tilde{y}(x)=y(z)$ and denote
\[p(x):=4x^3-g_2x-g_3=4\prod_{i=1}^3(x-e_i)=\wp'(z)^2,\]
where $e_i:=\wp(\frac{\omega_i}{2})$ with $\omega_3:=\omega_1+\omega_2=1+\tau$.
Then
\[\frac{dy}{dz}=\wp'(z)\frac{d\tilde{y}}{dx},\qquad
\frac{d^2y}{dz^2}=p(x)\frac{d^2\tilde{y}}{dx^2}+\frac12p'(x)\frac{d\tilde{y}}{dx},\]
\begin{align*}
\frac{d^3y}{dz^3}=\wp'(z)\left[p(x)\frac{d^3\tilde{y}}{dx^3}
+\frac32p'(x)\frac{d^2\tilde{y}}{dx^2}+12x\frac{d\tilde{y}}{dx}\right],
\end{align*}
so
\begin{equation}\label{ndd-ee1}Ly=\wp'(z)\tilde{L}\tilde{y},\end{equation}
where
\[\tilde{L}:=p(x)D^3+3(6x^2-\tfrac{g_2}{2})D^2+[(12-\alpha)x-B]D+\beta,\;\, D:=d/dx.\]
Namely $y(z)$ solves (\ref{3ode}) if and only if
$\tilde{L}\tilde{y}(x)=0$.

Denote $k:=(l+n-1)/2$ in the following proof. For each $i\in \{1,2,3\}$, we claim that (\ref{3ode}) has a solution $y_{i-1}(z)$ of the form
\begin{equation}\label{eq-eqeq}
y_{i-1}(z)=(\wp(z)-e_i)^{\frac12}\sum_{j=0}^{k}C_j (\wp(z)-e_i)^{j}\;\text{with}\;(C_k, C_{\frac{l}{2}-1})\neq (0,0),
\end{equation}
where $C_{-1}:=0$ if $l=0$.
Once this claim is proved, then $(y_0,y_1,y_2)$ forms a basis of solutions of (\ref{3ode}) (this fact that $y_0,y_1,y_2$ are linearly independent can be easily seen from the first proof of Theorem \ref{thm-nodd-k4} below), so $0$ is an apparent singularity.

To prove this claim, we denote
\[u:=\wp(z)-e_i=x-e_i,\quad \theta_i:=3e_i^2-g_2/4.\]
Then we can rewrite
\[p(x)=4u(u^2+3e_iu+\theta_i),\quad 6x^2-\tfrac{g_2}{2}=6u^2+12e_iu+2\theta_i,\]
and so rewrite $\tilde{L}$ as
\begin{align}\label{nnn-3eq-eq7}
\tilde{L}=&4u(u^2+3e_iu+\theta_i)D^3+(18u^2+36e_iu+6\theta_i)D^2\\
&+[(12-\alpha)u+(12-\alpha)e_i-B]D+\beta.\nonumber
\end{align}
Clearly $D=d/dx=d/du$, and $y_{i-1}(z)$ is a solution of (\ref{3ode}) if and only if
\begin{equation}\label{eqeqe}\tilde{y}:=(x-e_i)^{\frac12}\sum_{j=0}^{k}C_j(x-e_i)^j
=\sum_{j=0}^{k}C_j u^{j+\frac12}\end{equation}
satisfies $\tilde{L}\tilde{y}=0$.
Inserting (\ref{eqeqe}) into $\tilde{L}\tilde{y}$ leads to
\begin{equation}\label{nndd-1}\tilde{L}\tilde{y}=\sum_{j=-2}^{k}\Big(a_j C_j-[(j+\tfrac32)B +b_j e_i] C_{j+1}+4(j+2)(j+\tfrac32)(j+\tfrac{5}{2}) \theta_i C_{j+2}\Big)u^{j+\frac12},
\end{equation}
where
\begin{equation}\label{eqeqe2}C_{k+1}=C_{k+2}=C_{-1}=C_{-2}:=0,\end{equation}
\begin{align*}a_j:=&4(j+\tfrac12)(j-\tfrac12)(j-\tfrac32)+18(j+\tfrac12)(j-\tfrac12)
+(12-\alpha)(j+\tfrac12)+\beta\\
=&4(j+l+\tfrac{n+3}{2})(j-\tfrac{l+n-1}{2})(j-\tfrac{l}{2}+1)\\
=&4(j+l+\tfrac{n+3}{2})(j-k)(j-k+\tfrac{n+1}{2}),\end{align*}
and
\begin{align*}b_j:=&-[12(j+\tfrac32)(j+\tfrac12)(j-\tfrac12)+36(j+\tfrac32)(j+\tfrac12)
+(12-\alpha)(j+\tfrac32)]\\
=&-(j+\tfrac{3}{2})(12j^2+36j+27-3l^2-3l-2n-3ln-n^2).\end{align*}
Here we used the expressions (\ref{alpha-nl})-(\ref{beta-nl}) of $\alpha, \beta$.
Therefore, $\tilde{L}\tilde{y}=0$  if and only if
\begin{align}\label{eqeqe1}
&a_j C_j\\
=&[(j+\tfrac32)B +b_j e_i] C_{j+1}-4(j+2)(j+\tfrac32)(j+\tfrac{5}{2}) \theta_i C_{j+2},\;\, -2\leq j\leq k.\nonumber
\end{align}
Note that (\ref{eqeqe1}) holds automatically for $j\in \{k, -2\}$ because $a_{k}=0$ and (\ref{eqeqe2}).
Set the weights of $B$, $e_i$, $g_2$ to be $1, 1, 2$ respectively and recall $\theta_i=3e_i^2-g_2/4$.

{\bf Step 1.} We consider the case $n=1$.

Then $k=\frac{l}{2}$ and (\ref{remark-n1l2}) gives $P_{1, l}(B)=B$, so $B=0$. Furthermore, a direct computation gives
\begin{align*}b_j
&=-12(j+\tfrac{3}{2})(j+\tfrac{l}{2}+2)(j-\tfrac{l}{2}+1)\\
&=-12(j+\tfrac{3}{2})(j+k+2)(j-k+1)\quad\text{for }\;n=1,\end{align*}
so $a_{k-1}=b_{k-1}=0$, which implies that (\ref{eqeqe1}) with $B=0$ holds automatically for $j=k-1=\frac{l}{2}-1$.
Note $a_j\neq 0$ for $-1\leq j\leq k-2$.
Letting $j=k-2$ in (\ref{eqeqe1}) with $B=0$ leads to
\[C_{k-2}=\frac{b_{k-2}}{a_{k-2}} e_i C_{k-1}-\frac{4k(k-\tfrac12)(k+\tfrac{1}{2})}{a_{k-2}}(3e_i^2-\tfrac{g_2}{4}) C_{k}.\]
Then by induction for $j=k-2,k-3,\cdots, 0$ in (\ref{eqeqe1}) with $B=0$, we conclude the existence of
\[\tilde{P}_{j}(e_i,g_2),\quad \hat{P}_j(e_i, g_2)\in \mathbb{Q}[e_i, g_2]\]
that are of homogenous weights $k-1-j$, $k-j$ respectively, such that
\begin{equation}\label{eqeqe3}
C_j=\tilde{P}_{j}(e_i,g_2)C_{k-1}+\hat{P}_j(e_i, g_2)C_k,\quad 0\leq j\leq k-2.\end{equation}
Finally, by (\ref{eqeqe1}) with $B=0$, $j=-1$ and $C_{-1}=0$, we obtain
\begin{align}\label{eqeqe4}
0=&b_{-1}e_iC_0-3\theta_iC_1\\
=&\tilde{P}_{-1}(e_i, g_2)C_{k-1}
+[b_{-1}e_i\hat{P}_0(e_i, g_2)-3\theta_i\hat{P}_1(e_i, g_2)]C_{k}.\nonumber
\end{align}
where
\[\tilde{P}_{-1}(e_i, g_2):=b_{-1}e_i\tilde{P}_{0}(e_i,g_2)-3\theta_i\tilde{P}_{1}(e_i,g_2)\in
\mathbb{Q}[e_i, g_2]\]
is of homogenous weight $k$. Then (\ref{eqeqe1}) with $B=0$ holds if and only if (\ref{eqeqe3})-(\ref{eqeqe4}) hold.
There are two cases.

{\bf Case 1-1.} $\tilde{P}_{-1}(e_i, g_2)\neq 0$. This should hold except for at most discrete $\tau$'s.
Then by letting
\[C_k=1,\quad C_{k-1}=-\frac{b_{-1}e_i\hat{P}_0(e_i, g_2)-3\theta_i\hat{P}_1(e_i, g_2)}{\tilde{P}_{-1}(e_i, g_2)},\]
(such that (\ref{eqeqe4}) holds)
and $C_j$, $0\leq j\leq k-2$ be given by (\ref{eqeqe3}), we conclude that $y_{i-1}(z)$ given by (\ref{eq-eqeq}) is a solution of (\ref{3ode}) with $B=0$. Note that the local exponent of this $y_{i-1}(z)$ is $-1-2k=-l-1=-l-n$.

{\bf Case 1-2.} $\tilde{P}_{-1}(e_i, g_2)=0$. This should hold for at most discrete $\tau$'s.
Then by letting
\[C_k=0,\quad C_{k-1}=1,\]
(such that (\ref{eqeqe4}) holds) and $C_j$, $0\leq j\leq k-2$ be given by (\ref{eqeqe3}), we conclude that $y_{i-1}(z)$ given by (\ref{eq-eqeq}) is a solution of (\ref{3ode}) with $B=0$. Note that the local exponent of this $y_{i-1}(z)$ is $-1-2(k-1)=-l+1$ in this case.

In conclusion, (\ref{3ode}) with $B=0$ has a solution $y_{i-1}(z)$ of the form (\ref{eq-eqeq}) for $i=1,2,3$. This proves that $0$ is an apparent singularity for the case $n=1$.

{\bf Step 2.} We consider the general case $n\geq 3$.

Then $a_j\neq 0$ for $k-\frac{n+1}{2}+1\leq j\leq k-1$,
so by induction for $j=k-1,\cdots, k-\frac{n+1}{2}+1$ in (\ref{eqeqe1}), we easily conclude the existence of
\[\hat{P}_j(B)=\hat{P}_j(B;e_i,g_2)\in \mathbb{Q}[e_i, g_2][B]\]
with degree $k-j$ in $B$ and homogenous weight $k-j$, such that
\begin{equation}\label{qeqe3}
C_j=\hat{P}_j(B;e_i,g_2)C_k,\quad\text{for }\; 0\leq k-\tfrac{n+1}{2}+1\leq j\leq k-1.\end{equation}

By (\ref{qeqe3}), we see that the RHS of (\ref{eqeqe1}) with $j=k-\frac{n+1}{2}$ is $C_k$ multiplying a polynomial of degree $\frac{n+1}{2}$ in $B$ and homogenous weight $\frac{n+1}{2}$ in $\mathbb{Q}[e_i, g_2][B]$. Denote the corresponding monic polynomial by $\hat{P}_{i;n,l}(B)$, then there is $r_i\in\mathbb{Q}\setminus\{0\}$ such that
\begin{equation}\label{111111}\text{\it the RHS of (\ref{eqeqe1}) with $j=k-\tfrac{n+1}{2}$ equals  }C_kr_i\hat{P}_{i;n,l}(B).\end{equation}
At the moment it seems that $\hat{P}_{i;n,l}(B)\in \mathbb{Q}[e_i, g_2][B]$ might be different for $i=1,2,3$.
Since the LHS of (\ref{eqeqe1}) with $j=k-\frac{n+1}{2}$ is $0$ because of $a_{k-\frac{n+1}{2}}=0$, the key step is

{\bf Step 2-1.} We prove that
\begin{equation}
\label{keypoint} P_{n,l}(B)=\hat{P}_{i;n,l}(B),\quad \forall\;i\in \{1,2,3\}.
\end{equation}
Consequently, since $P_{n,l}(B)=0$, we see that (\ref{eqeqe1}) with $j=k-\frac{n+1}{2}=\frac{l}{2}-1$ holds automatically for arbitrary choices of $C_k$ and $C_{\frac{l}{2}-1}$.

Recall (\ref{eq-wp-ex}) that
\[u=x-e_i=\wp(z)-e_i=z^{-2}\Big(1-e_iz^2+\sum_{j=2}^{+\infty}B_{j}z^{2j}\Big),\]
Inserting this into
\[y_{i-1}(z)=\tilde{y}(x)=\sum_{j=0}^{k}C_j u^{j+\frac12}=\sum_{j=0}^{\frac{n+l-1}{2}}C_j u^{j+\frac12},\]
we can rewrite $y_{i-1}(z)$ as follows (see (\ref{nloc-sol}))
\[y_{i-1}(z)=\sum_{j=0}^{\infty}c_j z^{2j-n-l},\quad\text{where }\; c_0=C_k,\]
and so $Ly_{i-1}(z)$ is of the form (\ref{noddee}).

On the other hand, since (\ref{eqeqe1}) holds for $k-\frac{n+1}{2}+1\leq j\leq k$, we see from (\ref{nndd-1}) that (note $k-\frac{n+1}{2}=\frac{l}{2}-1$)
\begin{align*}
\tilde{L}\tilde{y}=&-C_kr_i\hat{P}_{i;n,l}(B) u^{\frac{l}{2}-1+\frac12}
+\sum_{j=-1}^{\frac{l}{2}-2}\Big(a_j C_j\\
&-[(j+\tfrac32)B +b_j e_i] C_{j+1}+4(j+2)(j+\tfrac32)(j+\tfrac{5}{2}) \theta_i C_{j+2}\Big)u^{j+\frac12}\\
=& -C_kr_i\hat{P}_{i;n,l}(B)z^{1-l}+\sum_{j=1}^{+\infty} d_j z^{2j+1-l}.
\end{align*}
Consequently, we apply (\ref{ndd-ee1}) to obtain
\begin{align*}
Ly_{i-1}(z)&=\wp'(z)\tilde{L}\tilde{y}\\
&=\left(-2z^{-3}+2B_2z+\cdots\right)\Big(-C_kr_i\hat{P}_{i;n,l}(B)z^{1-l}+\sum_{j=1}^{+\infty} d_j z^{2j+1-l}\Big)\\
&=2C_kr_i\hat{P}_{i;n,l}(B)z^{-l-2}+\sum_{j=1}^{+\infty} \tilde{d}_j z^{2j-l-2}.
\end{align*}
Comparing this with the expression (\ref{noddee}) of $Ly_{i-1}(z)$, we see that (\ref{nrec-app}) holds for $j\leq \frac{n+1}{2}-1$. Then (\ref{odd-appapp1}) holds for $Ly_{i-1}(z)$, so
\[-c_0r_0P_{n,l}(B)=2C_kr_i\hat{P}_{i;n,l}(B).\]
Since $C_k=c_0$ can be arbitrary at the moment, $r_0, r_i$ are nonzero rational numbers, and $P_{n,l}(B)$, $\hat{P}_{i;n,l}(B)$ are both monic polynomials of degree $\frac{n+1}{2}$ in $B$, we conclude $r_0=-2r_i$ and $P_{n,l}(B)=\hat{P}_{i;n,l}(B)$, i.e. (\ref{keypoint}) holds.

{\bf Step 2-2.} Consider the case $l=0$. Then Step 2-1 shows that (\ref{eqeqe1}) with $j=-1$ holds automatically for arbitrary choice of $C_k$. Thus by letting $C_k=1$, we immediately see that (\ref{3ode}) has a solution $y_{i-1}(z)$ of the form (\ref{eq-eqeq}) for $i=1,2,3$, so $0$ is an apparent singularity.

{\bf Step 2-3.} Consider the case $l= 2$. Then Step 2-1 shows that (\ref{eqeqe1}) with $j=0$ holds automatically for arbitrary choices of $C_k$ and $C_0$.
Consequently, by (\ref{eqeqe1}) with $j=-1$ and $C_{-1}=0$, we obtain
\begin{align*}
0=&(\tfrac{B}{2}+b_{-1}e_i)C_0-3\theta_iC_1\\
=&(\tfrac{B}{2}+b_{-1}e_i)C_{0}
-3\theta_i\hat{P}_1(B;e_i, g_2)C_k,
\end{align*}
where we used (\ref{qeqe3}). Then by letting
\[C_k=1,\quad C_{0}=\frac{3\theta_i\hat{P}_1(B;e_i, g_2)}{\tfrac{B}{2}+b_{-1}e_i}\quad \text{if }\;\tfrac{B}{2}+b_{-1}e_i\neq 0,\]
and
\[C_k=0,\quad C_{0}=1\quad \text{if }\;\tfrac{B}{2}+b_{-1}e_i=0,\]
we conclude again that (\ref{3ode}) has a solution $y_{i-1}(z)$ of the form (\ref{eq-eqeq}) for $i=1,2,3$, so $0$ is an apparent singularity.

{\bf Step 2-4.} Consider the general case $l\geq 4$. Then Step 2-1 shows that (\ref{eqeqe1}) with $j=k-\frac{n+1}{2}=\frac{l}{2}-1$ holds automatically for arbitrary choices of $C_k$ and $C_{\frac{l}{2}-1}$.
Since $a_j\neq 0$ for $0\leq j\leq k-\frac{n+1}{2}-1=\frac{l}{2}-2$,
then by induction for $j=\frac{l}{2}-2,\cdots, 0$ in (\ref{eqeqe1}), we easily conclude the existence of
\[\tilde{P}_{j}(B;e_i,g_2),\quad \hat{P}_j(B;e_i, g_2)\in \mathbb{Q}[e_i, g_2][B]\]
that are of homogenous weights $k-\frac{n+1}{2}-j$, $k-j$ respectively, such that
\begin{equation}\label{e111qeqe3}
C_j=\tilde{P}_{j}(B;e_i,g_2)C_{\frac{l}{2}-1}+\hat{P}_j(B;e_i, g_2)C_k,\quad 0\leq j\leq \tfrac{l}{2}-2.\end{equation}
Again by (\ref{eqeqe1}) with $j=-1$ and $C_{-1}=0$, we obtain
\begin{align*}
0&=(\tfrac{B}{2}+b_{-1}e_i)C_0-3\theta_iC_1\\
&=:\tilde{P}_{-1}(B;e_i, g_2)C_{\frac{l}{2}-1}
+\hat{P}_{-1}(B;e_i, g_2)C_{k}.
\end{align*}
Then by letting
\[C_k=1,\quad C_{\frac{l}{2}-1}=-\frac{\hat{P}_{-1}(B;e_i, g_2)}{\tilde{P}_{-1}(B; e_i, g_2)}\quad\text{if }\;\tilde{P}_{-1}(B; e_i, g_2)\neq 0,\]
and
\[C_k=0,\quad C_{\frac{l}{2}-1}=1\quad \text{if }\;\tilde{P}_{-1}(B; e_i, g_2)=0,\]
we conclude again that (\ref{3ode}) has a solution $y_{i-1}(z)$ of the form (\ref{eq-eqeq}) for $i=1,2,3$. This proves that $0$ is an apparent singularity.
The proof is complete.
\end{proof}

\begin{lemma}\label{lemma-distinctzero}
Let $n\geq 1$ be odd and $l\geq 0$ be even. Then
for any $\tau\in i\mathbb{R}_{>0}$, $P_{n,l}(B)$ has $\frac{n+1}{2}$ real distinct zeros.

Consequently, $P_{n,l}(B)$ has $\frac{n+1}{2}$ distinct zeros expect for $\lceil \frac{n^2-1}{24}\rceil$ numbers of $\tau$'s modulo $SL(2,\mathbb{Z})$. In particular, this result holds for any
 \begin{equation}\label{eq-rect}\tau\in \Pi:=\Big\{\tfrac{a\tilde{\tau}+b}{c\tilde{\tau}+d} \,|\, \bigl(\begin{smallmatrix}a & b\\
c & d\end{smallmatrix}\bigr)\in SL(2,\mathbb{Z}), \tilde{\tau}\in i\mathbb{R}_{>0}\Big\}.\end{equation}
\end{lemma}

\begin{proof}
Since $P_{1,l}(B)=B$, we only need to consider the case $n\geq 3$ and then (\ref{keypoint}) implies $P_{n,l}(B)=\hat{P}_{i;n,l}(B)$ for $i=1,2,3$. Here we use $i=3$ and prove this lemma via $\hat{P}_{3;n,l}(B)$.

The following proof is similar to that of Lemma \ref{lemma-real}.
First we consider the case $\tau\in i\mathbb{R}_{>0}$ and prove that $\hat{P}_{3;n,l}(B)$ has $\frac{n+1}{2}$ real distinct roots. This assertion will be proved by using the recursive formula (\ref{eqeqe1}) of $\hat{P}_{3;n,l}(B)$. This is the advantage of the recursive formula of $\hat{P}_{3;n,l}(B)$ comparing to that of $P_{n,l}(B)$ (i.e. it seems impossible to use the recursive formula (\ref{nrec-app}) of $P_{n,l}(B)$ to prove that $P_{n,l}(B)$ has $\frac{n+1}{2}$ real distinct roots).

Since $\tau\in i\mathbb{R}_{>0}$, it is well known that
\[e_1>e_3>e_2,\quad g_2>0,\]
and so
\[\theta_3=3e_3^2-g_2/4=(e_3-e_1)(e_3-e_2)<0.\]
Recall the recursive formula (\ref{eqeqe1}) with $i=3$:
\begin{align}\label{eqeqe1=i3}
&a_j C_j\\
=&[(j+\tfrac32)B +b_j e_3] C_{j+1}-4(j+2)(j+\tfrac32)(j+\tfrac{5}{2}) \theta_3 C_{j+2},\;\, -1\leq j\leq k,\nonumber
\end{align}
with $C_{-1}=C_{k+1}=C_{k+2}=0$ and
\[a_j=4(j+l+\tfrac{n+3}{2})(j-k)(j-k+\tfrac{n+1}{2}).\]
Recalling (\ref{qeqe3}), we let $C_k=1$ and then
\begin{equation}\label{qeqe3=i3}
C_j=C_j(B)=\hat{P}_j(B;e_3,g_2)\in\mathbb{R}[B]\end{equation}
is a polynomial of degree $k-j$ in $B$ for $k-\tfrac{n+1}{2}+1\leq j\leq k-1$.

Write $k_0:=k-\tfrac{n+1}{2}=\frac{l}{2}-1$ for convenience. Since $a_j<0$ for $k_0+1\leq j\leq k-1$ and $\theta_3<0$, we easily obtain the following properties for all $k_0+1\leq j\leq k-1$ from (\ref{eqeqe1=i3}):

{\bf (P1)} Up to a positive constant, the leading term in $C_{j}(B)$ is $(-1)^{k-j}B^{k-j}$.

{\bf (P2)} If $C_{j+1}(B)=0$ and $C_{j+2} (B)\neq 0$ for $B\in \mathbb{R}$, then
$C_{j} (B)C_{j+2} (B)<0$.
Then the same proof as Step 1 of Lemma \ref{lemma-real} implies that
for $k_0+1\leq j\leq k-1$, $C_j(B)$ has real distinct roots, denoted by $r_1^j<\cdots<r_{k-j}^j$, such that
\begin{equation}\label{root1=i3}
r_1^j<r_1^{j+1}<r_2^j<\cdots<r_{k-j-1}^{j}<r_{k-j-1}^{j+1}<r_{k-j}^j.
\end{equation}

Recall (\ref{111111}) that
\begin{equation}\label{Toda-eq-3-1==3}r_3\hat{P}_{3;n,l}(B)=[(k_0+\tfrac32)B +b_{k_0} e_3] C_{k_0+1}-4(k_0+2)(k_0+\tfrac32)(k_0+\tfrac{5}{2}) \theta_3 C_{k_0+2},\end{equation}
where $r_3\in\mathbb{Q}\setminus \{0\}$.
Recall {\bf (P1)} that
\[
\lim_{B\rightarrow-\infty} C_{k_0+2}(B) =+\infty, \; \lim_{B \rightarrow
+\infty} C_{k_0+2}(B) = (-1)^{k-k_0-2}\infty.
\]
By (\ref{root1=i3}) with $j=k_0+1$, we have
\[
C_{k_0+2}(r_i^{k_0+1})\sim (-1)^{i-1},\quad \forall i\in [1,k-k_0-1],
\]
and so it follows from (\ref{Toda-eq-3-1==3}) and $C_{k_0+1}(r_i^{k_0+1})=0$ that
\[r_3\hat{P}_{3;n,l}(r_i^{k_0+1})\sim C_{k_0+2}(r_i^{k_0+1})\sim (-1)^{i-1}, \quad \forall i\in [1,k-k_0-1].\]
On the other hand, up to a positive constant, the leading term of $r_3\hat{P}_{3;n,l}(B)$ is $(-1)^{k-k_0-1}B^{k-k_0}$, the same as that of $BC_{k_0+1}$, which implies
\[\lim_{B \rightarrow-\infty}r_3\hat{P}_{3;n,l}(B) =-\infty, \; \lim_{B \rightarrow
+\infty} r_3\hat{P}_{3;n,l}(B) = (-1)^{k-k_0-1} \infty.\]
Therefore,
$r_3\hat{P}_{3;n,l}(B)$ has $k-k_0=\frac{n+1}{2}$ real distinct roots.

This proves that for $\tau\in i\mathbb{R}_{>0}$, $P_{n,l}(B)$ has $\frac{n+1}{2}$ real distinct roots. Since $P_{n,l}(B)\in\mathbb{Q}[g_2(\tau), g_3(\tau)][B]$ is of homogenous weight $\frac{n+1}{2}$, the same proof as Step 3 of Lemma \ref{lemma-real} implies that $P_{n,l}(B)$ has $\frac{n+1}{2}$ distinct roots expect for $\lceil \frac{n^2-1}{24}\rceil$ numbers of $\tau$'s modulo $SL(2,\mathbb{Z})$, and in particular, this result holds for any $\tau\in\Pi$.
\end{proof}

As a consequence of the proof of Theorem \ref{thm-noddleven-apparent}, we have

\begin{theorem}[=Theorem \ref{thm-ode-mo}-(2)]\label{thm-nodd-k4}
Suppose $n\geq 1$ is odd, $l\geq 0$ is even and $P_{n,l}(B)=0$, i.e. $0$ is an apparent singularity of \eqref{3ode}. Then \eqref{3ode} has a basis of solutions $(y_0,y_1,y_2)$ of the form \eqref{eq-eqeq}, i.e.
\begin{equation}\label{eq-eqeq1}
y_{i-1}(z)=(\wp(z)-e_i)^{\frac12}\sum_{j=0}^{k_i}C_{i,j}(\wp(z)-e_i)^{j},\quad i=1,2,3,
\end{equation}
with $-1-2k_i\in \{-n-l, -l+1\}$ and $C_{i,k_i}=1$.

Furthermore, with respect to this basis, the monodromy matrices are given by
\begin{equation}\label{n-3eq-eq1}
N_1=\begin{pmatrix}1&&\\
&-1&\\
&&-1\end{pmatrix},\quad N_2=\begin{pmatrix}-1&&\\
&1&\\
&&-1\end{pmatrix}.
\end{equation}
Consequently, the monodromy group $M:=\langle N_1, N_2\rangle$ is the Klein four-group, and \eqref{3ode} has no nontrivial elliptic solutions.
\end{theorem}

\begin{remark}\label{rrrr}
Let $\tau\in i\mathbb{R}_{>0}$ in Theorem \ref{thm-nodd-k4}. Then Lemma \ref{lemma-distinctzero} proves $B\in\mathbb{R}$. Together with $e_1, e_2, e_3, g_2\in\mathbb{R}$ for $\tau\in i\mathbb{R}_{>0}$, it is easy to see from the proof of (\ref{eq-eqeq1}) or (\ref{eq-eqeq}) in Theorem \ref{thm-noddleven-apparent} that $C_{i,j}\in\mathbb{R}$ for all $i,j$.

On the other hand, it follows from the expression of $\wp(z)=\wp(z;\tau)$
\[\wp(z;\tau)=\frac{1}{z^{2}}+\sum_{(m,n)\in\mathbb{Z}^2\setminus
\{(0,0)\}}\left(  \frac{1}{(z-m-n\tau)^{2}}-\frac{1}{(m+n\tau)^{2}}\right) \]
and $\overline{\tau}=-\tau$ that $\overline{\wp(z)-e_i}=\wp(\overline{z})-e_i$ and so
\[\overline{(\wp(z)-e_i)^{\frac12}}=\pm(\wp(\overline{z})-e_i)^{\frac12}.\]
From here and (\ref{eq-eqeq1}), we conclude that $\overline{y_{i-1}(z)}=\pm y_{i-1}(\overline{z})$ and so
\begin{equation}\label{eff}
|y_{i-1}(z)|=|y_{i-1}(\overline{z})|,\quad  i=1,2,3.
\end{equation}
This property will play the crucial role in our proof of the symmetry of solutions for the $SU(3)$ Toda system \eqref{Toda} in a forthcoming work.
\end{remark}

\begin{proof}[Proof of Theorem \ref{thm-nodd-k4}] It suffices to prove (\ref{n-3eq-eq1}). Here we give two proofs.

{\bf The first proof:} This proof uses the expression of $y_j$'s.
Denote
\[\wp_i(z):=(\wp(z)-e_i)^{1/2}\quad\text{for  }i=1,2,3.\]
It is well known that $\wp_i(z)$ is an odd meromorphic function and satisfies
\begin{equation}\label{eq-k4}(\wp_1,\wp_2,\wp_3)(z+\omega_j)
=(\wp_1,\wp_2,\wp_3)(z)N_j,\quad j=1,2,\end{equation}
where $(N_1, N_2)$ is given by (\ref{n-3eq-eq1}).
From here and (\ref{eq-eqeq1}), we see that $y_{i-1}(z)$ satisfies the same transformation law (\ref{eq-k4}) as $\wp_i(z)$, i.e.
\[(y_0,y_1,y_2)(z+\omega_j)
=(y_0,y_1,y_2)(z)N_j,\quad j=1,2.\]
This implies that $(y_0, y_1, y_2)$ are \emph{linearly independent} and the monodromy matrices of (\ref{3ode}) under this basis $(y_0, y_1, y_2)$ are given by (\ref{n-3eq-eq1}), and so the monodromy group $\langle N_1, N_2\rangle$ is the Klein four-group.

{\bf The second proof:} This proof does not use the expression of $y_j$'s and so might be helpful for other problems.
Recalling (\ref{ndd-ee1}),
we consider the corresponding equation \begin{equation}\label{n-3eq-eq7}
D^3\tilde{y}+\frac{3(6x^2-\frac{g_2}{2})}{p(x)}D^2\tilde{y}
+\frac{(12-\alpha)x-B}{p(x)}D\tilde{y}+\frac{\beta}{p(x)}\tilde{y}=0\;\;\text{on }\mathbb{P}^1
\end{equation}
of (\ref{3ode}) under the double cover $\wp: E_{\tau}\to \mathbb{P}^1$, i.e. $x=\wp(z)$ and $\tilde{y}(x)=y(z)$.

Clearly under our assumption, any solution of (\ref{n-3eq-eq7}) has no logarithmic singularities either. Take a base point $x_{0}
\not \in \{e_1,e_2,e_3,\infty \}$ close to $\infty$ and consider the monodromy representation $\rho: \pi_1(\mathbb{C}\setminus\{e_1,e_2,e_3\},x_0)\to GL(3,\mathbb{C})$ of (\ref{n-3eq-eq7}). Let
$\gamma_{j}\in \pi_{1}(\mathbb{C}\backslash \{e_1,e_2,e_3\},x_{0})$ be a simple loop
encircling the singular point $e_j$ in the counterclockwise direction, and $\gamma_{\infty}$ be a simple
loop around $\infty$ clockwise such that%
\[
\gamma_{1}\gamma_{2}\gamma_{3}\gamma_{\infty}=Id\quad\text{in}\;
\pi_{1}(\mathbb{C}\backslash \{e_1,e_2,e_3\},x_{0}).
\]
Of course we require that all these loops do not intersect except at the base
point $x_{0}$. Let $\sigma_{j}=\rho(\gamma_j)$ and $\sigma_{\infty}=\rho(\gamma_{\infty})$ be the monodromy matrices
with respect to any basis of solutions of (\ref{n-3eq-eq7}). Then
$\sigma_1\sigma_2\sigma_3\sigma_{\infty}=I_3$, where $I_3=\operatorname{diag}(1,1,1)$ denotes the identity matrix.

Since the local exponents of (\ref{n-3eq-eq7}) at $\infty$ are $\frac{-n-l}{2}, \frac{1-l}{2}, \frac{n+2l+2}{2}$ (i.e. all in $\mathbb{Z}+\frac12$), and any solution of (\ref{n-3eq-eq7}) has no logarithmic singularities, it follows that $\sigma_{\infty}=-I_3$. Similarly, since the local exponents of (\ref{n-3eq-eq7}) at $e_j$ are $0,\frac12, 1$, it follows that
\begin{equation}\label{neqeq}\sigma_{j}=P_j \begin{pmatrix}1&&\\
&1&\\
&&-1\end{pmatrix} P_j^{-1},\quad j=1,2,3,\end{equation}
where $P_j\in GL(3,\mathbb{C})$ are connection matrices. Thus $\sigma_j^2=I_3$ for all $j$.

Now under the transformation $x=\wp(z)$, it is
easy to see that the fundamental cycle $z\to z+\omega_1$ (resp. $z\to z+\omega_2$) in $\pi_1(E_{\tau})$ is mapped to a simple loop in
$\pi_{1}(\mathbb{C}\backslash \{e_1,e_2,e_3\},x_{0})$ which separates $\{e_2,e_3\}$ from
$\{e_1,\infty \}$ (resp. separates $\{e_1,e_3\}$ from $\{e_2,\infty \}$).
Thus $N_1=\sigma_1\sigma_{\infty}=-\sigma_1$ and $N_2=\sigma_2\sigma_{\infty}=-\sigma_2$. In particular, $N_1^2=N_2^2=I_3$.

Clearly by $N_1=-\sigma_1$ and (\ref{neqeq}), we can take a basis of solutions of (\ref{3ode}) such that
\begin{equation}\label{n-3eq-eq2}N_1=\begin{pmatrix}1&&\\
&-1&\\
&&-1\end{pmatrix}.\end{equation}
Then $N_1N_2=N_2N_1$ (see (\ref{abelianm})) and $N_2^2=I_3$ lead to
\[N_2=\begin{pmatrix}d&&\\
&e&f\\
&g&h\end{pmatrix},\]
where $d=\pm 1$ and $\bigl(\begin{smallmatrix}e & f\\
g & h\end{smallmatrix}\bigr)^2=I_2$.
We claim that $d=-1$.

Indeed, if $d=1$, then by $N_2=-\sigma_2$ and (\ref{neqeq}), we see that the eigenvalues of $\bigl(\begin{smallmatrix}e & f\\
g & h\end{smallmatrix}\bigr)$ are $-1, -1$ and $\bigl(\begin{smallmatrix}e & f\\
g & h\end{smallmatrix}\bigr)$ can be diagonalized, so $\bigl(\begin{smallmatrix}e & f\\
g & h\end{smallmatrix}\bigr)=-I_2$, i.e.
\[
N_2=N_1=\begin{pmatrix}1&&\\
&-1&\\
&&-1\end{pmatrix}.\]
Consequently, it follows from $\sigma_j=-N_j$ for $j=1,2$ that
\[\sigma_1=\sigma_2=\begin{pmatrix}-1&&\\
&1&\\
&&1\end{pmatrix}.\]
This together with $\sigma_1\sigma_2\sigma_3\sigma_{\infty}=I_3$ and $\sigma_{\infty}=-I_3$ implies $\sigma_{3}=-I_3$, clearly a contradiction with (\ref{neqeq}) for $j=3$.

Therefore, $d=-1$ and then the eigenvalues of $\bigl(\begin{smallmatrix}e & f\\
g & h\end{smallmatrix}\bigr)$ are $1,-1$. So there is $P\in GL(2,\mathbb{C})$ such that
\[P \begin{pmatrix}
e&f\\
g&h\end{pmatrix}P^{-1}=\begin{pmatrix}
1&\\
&-1\end{pmatrix},\]
namely
\[
\begin{pmatrix}
1&\\
&P\end{pmatrix}N_1\begin{pmatrix}
1&\\
&P^{-1}\end{pmatrix}=\begin{pmatrix}1&&\\
&-1&\\
&&-1\end{pmatrix},\]
\[
\begin{pmatrix}
1&\\
&P\end{pmatrix}N_2\begin{pmatrix}
1&\\
&P^{-1}\end{pmatrix}=\begin{pmatrix}-1&&\\
&1&\\
&&-1\end{pmatrix}.\]
The proof is complete.
\end{proof}

\begin{remark}
It seems difficult to prove whether the $y_j(z)$'s in Theorem \ref{thm-nodd-k4} have only simple zeros or not via the ODE (\ref{3ode}). We can prove in another paper \cite{CL-Toda} that by using the deep connection between the ODE (\ref{3ode}) and the $SU(3)$ Toda system \eqref{Toda}, the following result can be easily proved by studying the bubbling phenomena of solutions of the Toda system:
\end{remark}

\begin{Theorem} \cite{CL-Toda}
Each $y_j(z)$ in \eqref{eq-eqeq1} has only simple zeros for any $j\in\{0,1,2\}$.
\end{Theorem}

\section{Monodromy theory for the case $n$ even}

\label{sec-monoth}

In this section we always assume that $n\geq 0$ is even, $l\in\mathbb{N}$ and $n+l\geq 1$.
We want to study the monodromy of (\ref{3ode}) and its dual equation (\ref{3ode-dual}) and prove Theorem \ref{thm-ode-mono}. The idea of the following proof seems new: We use the even elliptic solution of (\ref{3ode}) to transform the third order dual equation (\ref{3ode-dual}) to a second order ODE, and compute the monodromy of (\ref{3ode}) and the dual equation (\ref{3ode-dual}) by studying this new second order ODE.

First we have the following simple observation, which indicates the essential difference from the case $n$ odd that are studied in Sections \ref{sec-unitary}-\ref{sec-Klein}.

\begin{lemma}\label{Lem-app}
$0$ is always an apparent singularity of both \eqref{3ode} and its dual equation \eqref{3ode-dual}.
\end{lemma}

\begin{proof}
Recall that the local exponents of (\ref{3ode}) at $0$ are
\[-n-l, \quad -l+1,\quad n+2l+2.\]
If $l$ is odd, then $-n-l$ is odd and $-l+1$, $n+2l+2$ are both even, and the Forbenius's method implies the existence of a local solution of the form
\[y_{odd}(z)=z^{-n-l}\sum_{j=0}^{\infty}c_j z^{2j},\quad c_0=1,\]
with local exponent $-n-l$ (Note that there is no $z^{-l+1}$ and $z^{n+2l+2}$ terms in this $y_{odd}(z)$). Since Theorem \ref{even-elliptic-y1} says that (\ref{3ode}) has an even elliptic solution $y_{0}$ with local exponent $-l+1$, we conclude that all solutions of (\ref{3ode}) are meromorphic.

If $l$ is even, then $-l+1$ is odd and $-n-l$, $n+2l+2$ are both even. Again, (\ref{3ode}) has a local solution of the form
\[y_{odd}(z)=z^{-l+1}\sum_{j=0}^{\infty}c_j z^{2j},\quad c_0=1,\]
with local exponent $-l+1$ (Again there is no $z^{n+2l+2}$ term in this $y_{odd}(z)$). Since Theorem \ref{even-elliptic-y11} says that (\ref{3ode}) has an even elliptic solution $y_{0}$ with local exponent $-n-l$, we conclude that all solutions of (\ref{3ode}) are meromorphic.

Finally, it follows from \eqref{eq: data} that all solutions of the dual equation (\ref{3ode-dual}) are meromorphic.
\end{proof}

\begin{remark}\label{remark-evenelliptic}
 Theorems \ref{even-elliptic-y1}-\ref{even-elliptic-y11} and the proof of Lemma \ref{Lem-app} indicates that (\ref{3ode}) has a unique even elliptic solution $y_0(z)$ of the form
\[y_0(z)=\sum_{j=0}^{k}C_j(B)\wp(z)^j,\quad\text{where}\;\,
k=\begin{cases}\frac{l-1}{2}\quad\text{if $l$ odd},\\
\frac{n+l}{2}\quad\text{if $l$ even}.
\end{cases}\]
Here $C_{k}(B)=1$, $C_j(B)\in \mathbb{Q}[g_2,g_3][B]$ with degree $\deg C_j(B)=k-j$. Furthermore, $C_j(B)$ is homogenous of weight $k-j$, where the weights of $B, g_2, g_3$ are $1, 2, 3$ respectively.
\end{remark}

The rest of this section is to prove
that except at most $2(n+2l-k)+1$ choices of $B$'s, the generators $N_1, N_2$ of the monodromy group of (\ref{3ode}) can be diagonalized simultaneously; See Theorems \ref{Dual-Monodromy}-\ref{Dual-Monodromy1}. To this goal, we
denote
\begin{equation}\label{eq-s-4}p_1(z):=-(\alpha\wp(z)+B),\quad p_0(z):=\beta\wp'(z)\end{equation}
for convenience,
namely (\ref{3ode}) is written as
\begin{equation}\label{3ode-0}y'''+p_1y'+p_0y=0.\end{equation}

\begin{lemma}\label{lem-ode-2} Let $y_0(z)$ be the even elliptic solution of \eqref{3ode} in Remark \ref{remark-evenelliptic}.
Consider the following second order linear ODE
\begin{equation}\label{eq-sec-ode}
f''+\frac{3y_0'}{y_0}f'+\left(\frac{3y_0''}{y_0}+p_1\right)f=0.
\end{equation}
Then any solution of \eqref{eq-sec-ode} is of the form
\begin{equation}\label{eq-s-1}
f(z):=\frac{W(y,y_0)}{y_0^2}=\frac{y'y_0-yy_0'}{y_0^2},
\end{equation}
where $y(z)$ is a solution of \eqref{3ode}.
In particular, all solutions of \eqref{eq-sec-ode} are meromorphic.
\end{lemma}

\begin{proof}
Let $y(z)$ be a solution of (\ref{3ode}) and $f(z)$ be given by  (\ref{eq-s-1}).
Then
\begin{equation}\label{fc-y1}
y'y_0-yy_0'=y_0^2f,
\end{equation}
and so
\begin{equation}\label{fc-y2}
y''y_0-yy_0''=y_0^2f'+2y_0y_0'f,
\end{equation}
\begin{equation}\label{fc-y3}
y'''y_0-yy_0'''+y''y_0'-y'y_0''=y_0^2f''+4y_0y_0'f'+2y_0y_0''f+2y_0'^2f.
\end{equation}
Since $y'''=-p_1y'-p_0y$ and $y_0'''=-p_1y_0'-p_0y_0$, we have
\begin{equation}\label{fc-y4}y'''y_0-yy_0'''=-p_1(y'y_0-y_0'y)=-p_1y_0^2f.\end{equation}
Furthermore, by (\ref{fc-y1})-(\ref{fc-y2}) we easily obtain
\begin{equation}\label{fc-y5}y''y_0'-y'y_0''=y_0y_0'f'+2y_0'^2f-y_0y_0''f.\end{equation}
Inserting (\ref{fc-y4})-(\ref{fc-y5}) into (\ref{fc-y3}) leads to
 \[y_0^2f''+3y_0y_0'f'+(3y_0y_0''+p_1y_0^2)f=0,\]
 so $f(z)$
is a solution of (\ref{eq-sec-ode}). Since the set of such $f(z)$'s in (\ref{eq-s-1}) given by all solutions $y(z)$'s of (\ref{3ode}) is a linear space of dimension $2$, we conclude that such $f(z)$'s give all solutions of (\ref{eq-sec-ode}). The proof is complete.
\end{proof}

\begin{lemma}\label{lem-ode-1} Equation \eqref{eq-sec-ode} has a solution $f_1(z)$ such that
\begin{equation}\label{eq-ss-2}
f_1(z+\omega_j)=\lambda_jf_1(z),\quad \lambda_j\neq 0,\;\, j=1,2.
\end{equation}
Consequently, $f_2(z):=f_1(-z)$ is also a solution of \eqref{eq-sec-ode} and satisfies
\[f_2(z+\omega_j)=\lambda_j^{-1}f_2(z),\quad j=1,2,\]
and so
$f_1f_2$ is an even elliptic function.
\end{lemma}

\begin{proof}
Note that since $y_0(z)$ is even elliptic, then $\frac{3y_0''}{y_0}+p_1$ is even elliptic and $\frac{3y_0'}{y_0}$ is odd elliptic. In particular, each of $f(-z)$, $f(z+\omega_j)$ is a solution of (\ref{eq-sec-ode}) if $f(z)$ is.

Let $\tilde{N}_j\in GL(2,\mathbb{C})$ (not $SL(2,\mathbb{C})$ here) denote the monodromy matrix of (\ref{eq-sec-ode}) with respect to any basis $(\tilde{f}_1, \tilde{f_2})$ of solutions under the transformation $z\to z+\omega_j$, i.e.
\[(\tilde{f}_1, \tilde{f_2})(z+\omega_j)=(\tilde{f}_1, \tilde{f_2})(z)\tilde{N}_j,\quad j=1,2.\]
Then $\tilde{N}_1\tilde{N}_2=\tilde{N}_2\tilde{N_1}$, so $\tilde{N}_1$ and $\tilde{N}_2$ have a common eigenvector $\bigl(\begin{smallmatrix}a \\
b\end{smallmatrix}\bigr)\neq \bigl(\begin{smallmatrix}0 \\
0\end{smallmatrix}\bigr)$ such that $\tilde{N}_j\bigl(\begin{smallmatrix}a \\
b\end{smallmatrix}\bigr)=\lambda_j \bigl(\begin{smallmatrix}a \\
b\end{smallmatrix}\bigr)$ for $j=1,2$. Let $f_1(z)=a\tilde{f}_1(z)+b\tilde{f}_2(z)$,
then (\ref{eq-ss-2}) holds, i.e. $f_1(z)$ is a common eigenfunction of $\tilde{N}_1$ and $\tilde{N}_2$.
\end{proof}

By Lemmas \ref{lem-ode-2}-\ref{lem-ode-1}, there is a solution $\tilde{y}_1(z)$ of (\ref{3ode}) such that $f_1=\frac{W(\tilde{y}_1, y_0)}{y_0^2}$, i.e.
\begin{equation}\label{eq-s-10}
Y_1(z):=W(\tilde{y}_1, y_0)=y_0(z)^2f_1(z),
\end{equation}
and so
\begin{equation}\label{eq-s-11}
Y_2(z):=Y_1(-z)=W(\tilde{y}_2,y_0)=y_0(z)^2f_1(-z)=y_0(z)^2f_2(z),
\end{equation}
where $\tilde{y}_2(z):=-\tilde{y}_1(-z)$. Noting from Section \ref{sec-dual}.1 that $Y_1, Y_2$ are solutions of the dual equation (\ref{3ode-dual}), we immediately obtain

\begin{corollary}\label{coro-dual}
$Y_1(z), Y_2(z)$ given in \eqref{eq-s-10}-\eqref{eq-s-11} are common eigenfunctions of the monodromy matrices of the dual equation \eqref{3ode-dual}, namely
 \begin{equation}\label{eq-ss-02}
Y_1(z+\omega_j)=\lambda_jY_1(z),\quad \;\, j=1,2,
\end{equation}
\begin{equation}\label{eq-ss-03}
Y_2(z+\omega_j)=\lambda_j^{-1}Y_2(z),\quad \;\, j=1,2.
\end{equation}
In particular, the local exponents of $Y_1(z), Y_2(z)$ at $z=0$ must both be $-n-2l$.
\end{corollary}

\begin{proof}
Since $y_0(z)$ is even elliptic, it follows from Lemma \ref{lem-ode-1} that (\ref{eq-ss-02})-(\ref{eq-ss-03}) hold, which means that $Y_j(z)$ is \emph{elliptic of the second kind}\footnote{A meromorphic function $f(z)$ is called {\it elliptic of the second kind} if $f(z+\omega_j)=\lambda_jf(z)$ for some $\lambda_j\in\mathbb{C}\setminus\{0\}$, $j=1,2$. It is classical that non-constant functions of elliptic of the second kind must
have poles.}. Thus $Y_j(z)$ must have a pole at $z=0$ as a solution of the dual equation (\ref{3ode-dual}), and then it follows from (\ref{eq-s-7}) that the local exponent of $Y_j$ at $z=0$ is $-n-2l$.
\end{proof}

In view of Lemma \ref{lem-ode-1} and Corollary \ref{coro-dual}, a basic question is whether $f_1$ and $f_2$ (or equivalently $Y_1$ and $Y_2$) are linearly independent or not. To settle this question, we consider a transformation
\[\mathfrak{g}(z):=y_0^{\frac{3}{2}}f(z).\]
Then $f(z)$ solves (\ref{eq-sec-ode}) if and only if $g(z)$ solves
\begin{equation}\label{eq-s-2}
\mathfrak{g}''(z)=I(z)\mathfrak{g}(z),\quad \text{where}\;\, I(z):=-\frac{3y_0''}{2y_0}+\frac{3}{4}\left(\frac{y_0'}{y_0}\right)^2-p_1.
\end{equation}
Note that $I(z)$ is \emph{even elliptic} but any solution $\mathfrak{g}(z)$ of (\ref{eq-s-2}) has \emph{branch points} at simple zeros of $y_0(z)$.

Recalling $f_1, f_2$ in Lemma \ref{lem-ode-1}, we set
\begin{equation}\label{eq-s-20}\mathfrak{g}_1(z):=y_0^{\frac{3}{2}}f_1(z)=y_0^{-\frac{1}{2}}Y_1(z),\quad \mathfrak{g}_2(z):=y_0^{\frac{3}{2}}f_2(z)=y_0^{-\frac{1}{2}}Y_2(z),\end{equation}
then
\begin{equation}\label{eq-ss-101}
\Phi_{e}(z):=\mathfrak{g}_1(z)\mathfrak{g}_2(z)=y_0^3f_1(z)f_2(z)=y_0^{-1}Y_1Y_2
\end{equation}
is an even elliptic function. Furthermore, since $\mathfrak{g}_1(z)$ and $\mathfrak{g}_2(z)$ are solutions of (\ref{eq-s-2}), a direct computation shows that $\Phi_{e}(z)$ is a solution of the second symmetric product equation of (\ref{eq-s-2}):
\begin{equation}\label{eq-s-3}
\Phi'''-4I\Phi'-2I'\Phi=0.
\end{equation}
The following result shows that up to multiplying a constant, $\Phi_{e}(z)$ is the unique even elliptic solution of (\ref{eq-s-3}).

\begin{lemma}\label{lem-ode-3}
The dimension of even elliptic solutions of \eqref{eq-s-3} is one.
\end{lemma}

\begin{proof} Recall Remark \ref{remark-evenelliptic} that $y_0(z)=z^{-2k}(1+O(|z|^2))$ and $p_1(z)=-\alpha\wp(z)-B=-\alpha z^{-2}(1+O(|z|^2))$ at $z=0$. Inserting these  and (\ref{alpha-nl}) into (\ref{eq-s-2}), a direct computation shows that the local exponent $\rho$ of (\ref{eq-s-2}) at $z=0$ satisfies
\begin{align*}\rho(\rho-1)&-[(n+l)^2+(l+2)(n+l)+l(l+1)]\\
&+3k(2k+1)-3k^2=0,\end{align*}
so the local exponents of (\ref{eq-s-2}) at $z=0$ are
\[\rho_1=-n-l-\tfrac{l+1}{2},\;\; \rho_2=n+l+\tfrac{l+1}{2}+1,\;\;\text{if $l$ is odd, i.e. $k=\tfrac{l-1}{2}$},\]
\[\rho_1=-\tfrac{n+l}{2}-l,\;\; \rho_2=\tfrac{n+l}{2}+l+1,\;\;\text{if $l$ is even, i.e. $k=\tfrac{n+l}{2}$}.\]
Consequently, the local exponents of (\ref{eq-s-3}) at $z=0$ are $2\rho_1, 1, 2\rho_2$, so
 the dimension of the space of even solutions to %
(\ref{eq-s-3}) is $2$.

Suppose by contradiction that the dimension of even
elliptic solutions to (\ref{eq-s-3}) is $2$, then any even solution
of (\ref{eq-s-3}) must be elliptic.
On the other hand, (\ref{eq-s-2}) has local
solutions of the following form at $0$:
\begin{equation}\label{eq-s-5}
\hat{\mathfrak{g}}_{1}(z)=z^{\rho_1}\bigg(1+\sum_{j=1}^{\infty}a_{j}z^{2j}\bigg)  ,%
\text{ \ }\hat{\mathfrak{g}}_{2}(z) =z^{\rho_2}\bigg(1+\sum_{j=1}^{\infty}b_{j}z^{2j}%
\bigg)  .
\end{equation}
Then $\hat{\mathfrak{g}}_j(z)^2$, $j=1,2$, are even solutions of (\ref{eq-s-3}) and
hence even elliptic functions by our assumption. Define
\begin{equation}  \label{2-15}
\hat{f}_j(z):=\hat{\mathfrak{g}}_j(z)y_0(z)^{-\frac{3}{2}},\quad j=1,2,
\end{equation}
then $\hat{f}_j(z)$ are linearly independent solutions of (\ref{eq-sec-ode}) and so meromorphic, and $\hat{f}_j(z)^2=\hat{\mathfrak{g}}_j^2y_0^{-3}$ are even elliptic functions.
On the other hand, there are $\SM{a_i}{b_i}{c_i}{d_i}\in GL(2,\mathbb{C})$ such that
\begin{equation*}
\begin{pmatrix}
\hat{f}_{1}( z+\omega_{i}) \\
\hat{f}_{2}( z+\omega_{i})%
\end{pmatrix}
=
\begin{pmatrix}
a_i & b_i \\
c_i & d_i%
\end{pmatrix}
\begin{pmatrix}
\hat{f}_{1}( z) \\
\hat{f}_{2}( z)%
\end{pmatrix}%
,\quad i=1,2.
\end{equation*}
Then%
\begin{align*}
\hat{f}_{1}(z) ^{2} =\hat{f}_{1}( z+\omega _{i})^{2}=a_i^{2}\hat{f}
_{1}( z) ^{2}+2a_ib_i\hat{f}_{1}(z) \hat{f}_{2}(z) +b_i^{2}\hat{f}
_{2}(z)^{2}
\end{align*}
yields $b_i=0$ and $a_i=\pm 1$. Similarly, we could use $
\hat{f}_2$ to obtain $c_i=0$ and $d_i=\pm 1$, i.e.
\begin{equation}\label{eq-3ord-1}\hat{f}_{1}( z+\omega _{i})=a_i \hat{f}_{1}(z),\quad \hat{f}_{2}( z+\omega _{i})=d_i\hat{f}_{2}( z),\quad i=1,2.\end{equation}
Note that the Wronskian
\[W(\hat{f}_1,\hat{f}_2)=\hat{f}_1'\hat{f}_2-\hat{f}_2'\hat{f}_1
=\frac{W(\hat{\mathfrak{g}}_1,\hat{\mathfrak{g}}_2)}{y_0^3}\]
is an elliptic function because $W(\hat{\mathfrak{g}}_1,\hat{\mathfrak{g}}_2)=\hat{\mathfrak{g}}_1'\hat{\mathfrak{g}}_2-\hat{\mathfrak{g}}_2'\hat{\mathfrak{g}}_1$ is a nonzero constant by applying that $\hat{\mathfrak{g}}_1, \hat{\mathfrak{g}}_2$ are linearly independent solutions of (\ref{eq-s-2}). This, together with $W(\hat{f}_1,\hat{f}_2)(z+\omega_i)=a_id_iW(\hat{f}_1,\hat{f}_2)(z)$, implies $a_id_i=1$ and so $(\hat{f}_{1}\hat{f}_{2})(z+\omega_i)=(\hat{f}_{1}\hat{f}_{2})(z)$, i.e.
\[\hat{f}_{1}(z)\hat{f}_{2}( z)=\frac{\hat{\mathfrak{g}}_1(z)\hat{\mathfrak{g}}_2(z)}{y_0(z)^3}\quad\text{\it is odd and elliptic},\]
where the fact $\hat{\mathfrak{g}}_1(z)\hat{\mathfrak{g}}_2(z)$ being odd is used.
Since Lemma \ref{lem-ode-2} says that there are solutions $\hat{y}_j(z)$ of (\ref{3ode}) such that
\[\hat{f}_j(z)=\frac{W(\hat{y}_j,y_0)}{y_0^{2}},\] it follows that
\[
W(\hat{y}_1,y_0)W(\hat{y}_2,y_0)=y_0\hat{\mathfrak{g}}_1(z)\hat{\mathfrak{g}}_2(z)\quad\text{\it is odd and elliptic}.
\]
Note from (\ref{eq-s-5}) that the local exponent of $W(\hat{y}_1,y_0)W(\hat{y}_2,y_0)$ at $z=0$ is $1-2k$. Recalling that $W(\hat{y}_1,y_0)$, $W(\hat{y}_2,y_0)$ are both solutions of the dual equation (\ref{3ode-dual}), we conclude from (\ref{eq-s-7}) that one of $W(\hat{y}_1,y_0)$ and $W(\hat{y}_2,y_0)$, say $W(\hat{y}_1,y_0)$ for example, has the local exponent $n+l+2$ at $z=0$ if $2k=l-1$ (resp. $l+1$ if $2k=n+l$), and $W(\hat{y}_2,y_0)$ has the local exponent $-n-2l$ at $z=0$. In conclusion, $W(\hat{y}_1, y_0)=y_0^{2}\hat{f}_1(z)$ is holomorphic and has a zero of finite order (i.e. $n+l+2$ if $2k=l-1$, or $l+1$ if $2k=n+l$) at $z=0$, and satisfies (recall (\ref{eq-3ord-1}))
\[W(\hat{y}_1, y_0)(z+\omega_i)=a_iW(\hat{y}_1, y_0),\quad a_i=\pm 1, \;\,i=1,2,\]
so the Liouville theorem implies $W(\hat{y}_1, y_0)\equiv 0$, i.e. $\hat{f}_1(z)\equiv 0$, clearly a contradiction. This proves that the dimension of even elliptic solutions of (\ref{eq-s-3}) is one. The proof is complete.
\end{proof}

By (\ref{eq-ss-101}) and Corollary \ref{coro-dual},
\begin{equation}\label{eq-s-12}
F(z):=Y_1(z)Y_2(z)=y_0(z)\Phi_{e}(z)
\end{equation}
is even elliptic and has only a pole at $z=0$ with order $2(n+2l)$, so
up to multiplying a constant,
\begin{equation}\label{eq-s-13}
F(z)=\sum_{j=0}^{m}\beta_j \wp(z)^j,\quad \text{where }\;m:=n+2l,\;\,\beta_m=1.
\end{equation}

\begin{lemma}
For any $0\leq j\leq m$, $\beta_j=\beta_j(B)\in \mathbb{Q}[g_2, g_3][B]$ is homogenous of weight $m-j$, where the weights of $B, g_2, g_3$ are $1, 2, 3$ respectively, i.e. $\deg_B \beta_j\leq m-j$. Furthermore, $\deg_B\beta_0=m$, so by multiplying a rational number $r\neq 0$ to both $Y_1(z)$ and $Y_2(z)$, we may always assume that $\beta_0=B^m+\cdots$ is monic and then $\beta_m=r^2\in\mathbb{Q}\setminus\{0\}$.
\end{lemma}

\begin{proof} We divide our proof into several steps.

{\bf Step 1.} We show that $F(z)$ solves the following equation
\begin{align}\label{eq-s-103}
y_0^2F'''&-3y_0y_0'F''+[3(y_0')^2+3y_0y_0''+4p_1y_0^2]F'\\
&+[-6y_0'y_0''-6p_1y_0y_0'+2(p_1'-p_0)y_0^2]F=0.\nonumber
\end{align}

Indeed, $\Phi_e=\frac{F}{y_0}$ gives
\[\Phi_e'=\frac{F'}{y_0}-\frac{y_0'}{y_0^2}F,\]
\begin{align*}
\Phi_e'''=\frac{F'''}{y_0}-\frac{3y_0'}{y_0^2}F''+\frac{6(y_0')^2-3y_0y_0''}{y_0^3}F'
+\frac{-6(y_0')^3+6y_0y_0'y_0''-y_0^2y_0'''}{y_0^4}F.
\end{align*}
Inserting these and
\[I=-\frac{3y_0''}{2y_0}+\frac{3}{4}\left(\frac{y_0'}{y_0}\right)^2-p_1,\]
\[I'=-\frac{3y_0'''}{2y_0}+3\frac{y_0'y_0''}{y_0^2}-\frac{3(y_0')^3}{2y_0^3}-p_1',\]
into (\ref{eq-s-3}) and using $y_0'''=-p_1y_0'-p_0y_0$, we easily obtain (\ref{eq-s-103}).

{\bf Step 2.} We show that $\beta_j=\beta_j(B)\in \mathbb{Q}[g_2, g_3][B]$ is homogenous of weight $m-j$. In particular, $\deg_B\beta_j\leq m-j$.

Write $x=\wp(z)$ for convenience, then $x':=\frac{dx}{dz}=\wp'(z)$. Recall Remark \ref{remark-evenelliptic} that
\[y_0(z)=\sum_{j=0}^{k}C_jx^j,\quad\text{where}\;\,k=\begin{cases}\frac{l-1}{2}\quad\text{if $l$ odd},\\
\frac{n+l}{2}\quad\text{if $l$ even}.
\end{cases}\]
Here $C_{k}=1$, $C_j=C_j(B)\in \mathbb{Q}[g_2,g_3][B]$ is homogenous of weight $k-j$ with degree $\deg_B C_j=k-j$.
Also recall (\ref{eq-s-13}) that
\[F=\sum_{j=0}^{m}\beta_j x^j,\quad \beta_m=1.\]
By
\[(\wp')^2=4\wp^3-g_2\wp-g_3=4x^3-g_2x-g_3,\]
\[\wp''=6\wp^2-\tfrac{g_2}{2}=6x^2-\tfrac{g_2}{2},\quad \wp'''=12\wp'\wp=12\wp'x,\]
we have
\[\frac{y_0'}{\wp'}=\sum_{j=0}^{k}jC_jx^{j-1},\quad \frac{F'}{\wp'}=\sum_{j=0}^{m}j\beta_jx^{j-1},\]
\begin{align*}y_0''=
(4x^3-g_2x-g_3)\sum_{j=0}^{k}j(j-1)C_jx^{j-2}+(6x^2-\tfrac{g_2}{2})\sum_{j=0}^{k}jC_jx^{j-1},
\end{align*}
\begin{align*}F''&=
(4x^3-g_2x-g_3)\sum_{j=0}^{m}j(j-1)\beta_jx^{j-2}+(6x^2-\tfrac{g_2}{2})
\sum_{j=0}^{m}j\beta_jx^{j-1}\\
&=\sum_{j=0}^{m}[(4j^2+2j)\beta_j-(j+2)(j+1)g_2\beta_{j+2}]x^{j+1}\\
&\quad-
\sum_{j=-1}^{m}[\tfrac{j+2}{2}g_2\beta_{j+2}+(j+3)(j+2)g_3\beta_{j+3}]x^{j+1},
\end{align*}
{\allowdisplaybreaks
\begin{align*}
\frac{F'''}{\wp'}=&(4x^3-g_2x-g_3)\sum_{j=0}^{m}j(j-1)(j-2)\beta_jx^{j-3}\\
&+3(6x^2-\tfrac{g_2}{2})\sum_{j=0}^{m}j(j-1)\beta_jx^{j-2}+12\sum_{j=0}^{m}j\beta_jx^{j}\\
=&\sum_{j=0}^m\Big\{[4j(j-1)(j-2)+18j(j-1)+12j]\beta_j\\&-(j+2)(j+1)(j+\tfrac{3}{2})
g_2\beta_{j+2}
-(j+3)(j+2)(j+1)g_3\beta_{j+3}\Big\}x^{j},
\end{align*}}
where $\beta_{m+3}=\beta_{m+2}=\beta_{m+1}:=0$.

Inserting these and
\[p_1=-\alpha x-B,\quad p_0=\beta\wp'\]
 into (\ref{eq-s-103}) and eliminating the common factor $\wp'$, we obtain
 \[\sum_{l=0}^{2k+m}\xi_l x^l:=\frac{\text{LHS of (\ref{eq-s-103})}}{\wp'}=0,\]
where $\xi_l\in \mathbb{Q}[g_2,g_3,C_j,\beta_j,B]$. Thus $\xi_l=0$ for all $0\leq l\leq 2k+m$. On the other hand, for $l=2k+j$ with $0\leq j\leq m$, a direct computation shows that we can write $\xi_{2k+j}$ as linear combinations of $\beta_j, \beta_{j+1}, \cdots$, i.e.
\begin{align}\label{eq-s-16}
0=\xi_{2k+j}=\vartheta_j \beta_j+\sum_{i>j}\eta_{j,i}\beta_i,
\end{align}
where $\eta_{j,i}\in \mathbb{Q}[g_2,g_3, B, C_k, C_{k-1},\cdots, C_{0}]=\mathbb{Q}[g_2,g_3][B]$ and (note $C_k=1$)
\begin{align}\label{eq-s-15}
\vartheta_j=&4j(j-1)(j-2)+18j(j-1)+12j-3k(4j^2+2j)\\
&+j(24k^2+6k-4\alpha)+[-6k(4k^2+2k)+(6k-2)\alpha-2\beta].\nonumber
\end{align}
Recall (\ref{alpha-nl})-(\ref{beta-nl}) that
\[
\alpha=(n+l)^2+(l+2)(n+l)+l(l+1),
\;\,
\beta=\frac{(l-1)(n+l)(n+2l+2)}{2}.
\]
Inserting these into (\ref{eq-s-15}) and noting $m=n+2l$,
we easily compute that for $l$ odd, i.e. $k=\frac{l-1}{2}$, there holds
\[\vartheta_j=4(j-m)(j-\tfrac{l-2}{2})(j+m-l+2),\]
and for $l$ even, i.e. $k=\frac{n+l}{2}$, there holds
\[\vartheta_j=4(j-m)(j-\tfrac{m-l-1}{2})(j+l+1).\]
Therefore, in both cases we have $\vartheta_m=0$ and
\[\vartheta_j\neq 0,\quad j=0,1,\cdots, m-1.\]
Since $\beta_m=1$ and $\beta_{m+3}=\beta_{m+2}=\beta_{m+1}=0$, so (\ref{eq-s-16}) holds automatically for $j=m$. By induction on (\ref{eq-s-16}) for $j=m-1, m-2, \cdots, 1, 0$, we easily conclude that $\beta_j\in \mathbb{Q}[g_2,g_3][B]$ for all $j=m-1,\cdots, 0$.

Finally, by considering the weight of $x$ being $1$, we see that $y_0$ is homogeneous of weight $k$, $\frac{y_0'}{\wp'}$ is homogeneous of weight $k-1$ and $y_0''$ is homogenous of weight $k+1$.

Clearly $\beta_m=1$ is homogenous of weight $0$. Suppose for some $j\leq m$ there holds that $\beta_i$ is homogenous of weight $m-i$ for all $i\in [j, m]$. Then $\sum_{i>j}\beta_i x^i$ (as a partial sum in the expression of $F$) is homogenous of weight $m$. Consequently, as parts of the coefficient $\xi_{2k+j}$ of $x^{2k+j}$, it is easy to see that $\sum_{i>j}\eta_{j,i}\beta_i$ is homogenous of weight $(2k+m)-(2k+j)=m-j$. Therefore, we see from (\ref{eq-s-16}) that $\beta_j$ is homogenous of weight $m-j$. This induction argument proves that for all $j\in [0,m]$, $\beta_j\in \mathbb{Q}[g_2,g_3][B]$ is homogenous of weight $m-j$.

{\bf Step 3.} We show that $\deg_B\beta_0=m$.

Suppose by contradiction that $\deg_B\beta_0=\tilde{m}<m$. Fix any $\tau$ such that the leading coefficient $c(g_2,g_3)$ of $\beta_0=c(g_2,g_3)B^{\tilde{m}}+\cdots$ is not zero.
Denote
\[k_j:=\deg_B\beta_j(B)\leq m-j\quad\text{for }j\in [1, m],\quad k_m=0. \]

Fix any $\varepsilon\in (0,1)$ such that $m^3\varepsilon<1$ and let $x=\wp(z)=B^{1-\varepsilon}$. Recall
\[F=x^m+\beta_{m-1}(B)x^{m-1}+\cdots\beta_{1}(B)x+\beta_{0}(B).\]
Then for $|B|$ large,
\[x^m\sim B^{m(1-\varepsilon)},\quad \beta_{0}(B)\sim B^{\tilde{m}},\quad m(1-\varepsilon)>\tilde{m},\]
\[\beta_{j}(B)x^j\sim B^{k_j+j(1-\varepsilon)},\quad\forall j\in [1,m],\]
where $f\sim B^s$ denotes $f=cB^s(1+o(1))$ with $c\neq 0$ as $|B|\to \infty$. Since
$k_j+j(1-\varepsilon)\neq k_i+i(1-\varepsilon)$ for $j\neq i$, there is a unique $j_0\in [1, m]$ such that
\[k_{j_0}+j_0(1-\varepsilon)>\max_{j\in[1,m]\setminus\{j_0\}} (k_j+j(1-\varepsilon))>\tilde{m},\]
namely,
\[F\sim B^{k_{j_0}+j_0(1-\varepsilon)}.\]
Similarly,
\[\frac{F'}{\wp'}\sim B^{k_{j_0}+(j_0-1)(1-\varepsilon)},\quad F''=O\left( B^{k_{j_0}+(j_0+1)(1-\varepsilon)}\right),\]
\[\frac{F'''}{\wp'}=O\left( B^{k_{j_0}+j_0(1-\varepsilon)}\right).\]

On the other hand, since $y_0=\sum_{j=0}^k C_j(B)x^j$ with $\deg_B C_j=k-j$, so
\[y_0\sim B^{k},\quad \frac{y_0'}{\wp'}\sim B^{k-1},\quad y_0''\sim B^{k-1+2(1-\varepsilon)}.\]
Inserting these and $p_1\sim B$ into $\frac{\text{LHS of (\ref{eq-s-103})}}{\wp'}$, it is easy to see that
\[4p_1y_0^2\frac{F'}{\wp'}\sim B^{2k+k_{j_0}+j_0(1-\varepsilon)+\varepsilon},\]
and all the other terms of $\frac{\text{LHS of (\ref{eq-s-103})}}{\wp'}$ is at most $O(B^{2k+k_{j_0}+j_0(1-\varepsilon)})$, namely
\[0=\frac{\text{LHS of (\ref{eq-s-103})}}{\wp'}=cB^{2k+k_{j_0}+j_0(1-\varepsilon)+\varepsilon}
(1+O(|B|^{-\varepsilon}))\]
for some $c\neq 0$ as $|B|\to \infty$, clearly a contradiction.

This completes the proof.
\end{proof}

\begin{lemma}\label{lem-ode-4} Recall \eqref{eq-ss-101} that $\Phi_e=\mathfrak{g}_1\mathfrak{g}_2$.
Define
\begin{equation}\label{eq-s-18}
Q_{n,l}(B):=I(z)\Phi_e(z)^2+\frac{\Phi_e'(z)^2}{4}-\frac{\Phi_e(z)\Phi_e''(z)}{2}.
\end{equation}
Then $Q_{n,l}(B)=W(\mathfrak{g}_1,\mathfrak{g}_2)^2/4\in \mathbb{Q}(g_2,g_3)[B]$ is a monic polynomial of $B$ with degree \[\deg Q_{n,l}=2(m-k)+1=\begin{cases}2n+3l+2 \quad\text{if $l$ is odd}\\
n+3l+1\quad\text{if $l$ is even}
\end{cases},\] where $W(\mathfrak{g}_1,\mathfrak{g}_2)=\mathfrak{g}_1'\mathfrak{g}_2-\mathfrak{g}_1\mathfrak{g}_2'$ is the Wronskian.
\end{lemma}

\begin{proof}
Since $\Phi_e(z)$ solves (\ref{eq-s-3}), the derivative of the RHS of (\ref{eq-s-18}) with respect to $z$ is clearly $0$, so $Q_{n,l}(B)$ is independent of $z$. Since
\begin{align}
\Phi_e&=\frac{F}{y_0}=\frac{\sum_{j=0}^m \beta_j x^j}{\sum_{j=0}^k C_j x^j}\nonumber\\
&=\frac{B^m+F_{m-1}(x)B^{m-1}+\cdots+ F_1(x)B+F_0(x)}{B^k+\varphi_{k-1}(x)B^{k-1}+\cdots+\varphi_1(x)B+\varphi_0(x)},
\end{align}
\[I=-\frac{3y_0''}{2y_0}+\frac{3}{4}\left(\frac{y_0'}{y_0}\right)^2-p_1
=B+\alpha x-\frac{3y_0''}{2y_0}+\frac{3}{4}\left(\frac{y_0'}{y_0}\right)^2,\]
where $\beta_j, C_j\in \mathbb{Q}[g_2,g_3,B]$ imply $F_j(x), \varphi_j(x)\in \mathbb{Q}[g_2, g_3, x]$, we conclude that $Q_{n,l}(B)\in \mathbb{Q}(g_2,g_3,B)$.

Recall (\ref{eq-ss-101}) that $\Phi_e=\mathfrak{g}_1\mathfrak{g}_2$, where $\mathfrak{g}_j''=I\mathfrak{g}_j$. Then the Wronskian
$W=W(\mathfrak{g}_1, \mathfrak{g}_2)=\mathfrak{g}_1'\mathfrak{g}_2-\mathfrak{g}_1\mathfrak{g}_2'$
satisfies
\[
\frac{\mathfrak{g}_{1}^{\prime}}{\mathfrak{g}_{1}}=\frac{\Phi_e^{\prime}+W}%
{2\Phi_e},\quad \frac{\mathfrak{g}_{2}^{\prime}}{\mathfrak{g}_{2}}=\frac{\Phi_e^{\prime
}-W}{2\Phi_e},
\]
which easily implies
\[
\frac{\Phi_e^{\prime \prime}}{2\Phi_e}-\frac{\Phi_e^{\prime}\pm W}
{2\Phi_e^{2}}\Phi_e^{\prime}=I-\left(  \frac{\Phi_e^{\prime
}\pm W}{2\Phi_e}\right)  ^{2}.
\]
From here we immediately obtain
\begin{equation}
\frac{W(\mathfrak{g}_1, \mathfrak{g}_2)^{2}}{4}=\frac{W^2}{4}=I\Phi_e^{2}+\frac{(\Phi_e^{\prime})^2}{4}-\frac
{\Phi_e\Phi_e^{\prime \prime}}{2}=Q_{n,l}(B).\label{eqWron}%
\end{equation}
In particular, it follows that $Q_{n,l}(B)\neq \infty$ for any $B$, so $Q_{n,l}(B)\in \mathbb{Q}(g_2, g_3)[B]$ is a polynomial of $B$.

To determine its degree, we fix any $z$ such that $x=\wp(z)$ is finite. Then as $|B|$ large, we easily deduce
\[\Phi_e=B^{m-k}(1+o(B^{-1})),\quad \Phi_e'=O(B^{m-k-1}), \quad \Phi_e''=O(B^{m-k-1}),\]
and $I=B(1+o(B^{-1}))$,
so we see from (\ref{eq-s-18}) that
\[Q_{n,l}(B)=B^{2(m-k)+1}(1+o(B^{-1})),\quad \text{as $|B|$ large},\]
i.e. $Q_{n,l}(B)$ is a monic polynomial of $B$ with $\deg_B Q_{n,l}=2(m-k)+1$.
\end{proof}

Recall that $N_j$ denotes the monodromy matrix with respect to any basis $(y_0, y_1, y_2)$ of the original equation (\ref{3ode}) under the transformation $z\to z+\omega_j$, i.e.
\[(y_0, y_1, y_2)(z+\omega_j)=(y_0, y_1, y_2)(z)N_j,\quad j=1,2,\]
here we use $M_j$ to denote the monodromy matrix with respect to any basis $(Y_0, Y_1, Y_2)$ of the dual equation (\ref{3ode-dual}) under the transformation $z\to z+\omega_j$, i.e.
\[(Y_0, Y_1, Y_2)(z+\omega_j)=(Y_0, Y_1, Y_2)(z)M_j,\quad j=1,2.\]
Now we can state
the main results of this section as follows, which implies Theorem \ref{thm-ode-mono}.

\begin{theorem}\label{Dual-Monodromy} Recall $y_0(z)$ in Remark \ref{remark-evenelliptic} and $Y_0(z)$ in Theorem \ref{even-elliptic-Y1}, and let $Y_1(z), Y_2(z)$ be in Corollary \ref{coro-dual}. Let $B\in\mathbb{C}$ such that $Q_{n,l}(B)\neq 0$. Then
\begin{itemize}
\item[(1)] $Y_0, Y_1, Y_2$ form a basis of solutions of the dual equation \eqref{3ode-dual}, and under this basis, the monodromy matrices $M_j$ are given by
    \begin{equation}\label{eq-s-21}M_j=\operatorname{diag}(1, \lambda_j, \lambda_j^{-1}),\quad j=1,2.\end{equation}
    Furthermore, $\lambda_j\neq \pm 1$ for at least one $j\in \{1,2\}$.
\item[(2)] $y_0$,  $y_1:=W(Y_2, Y_0)$, $y_2:=W(Y_0, Y_1)$ form a basis of solutions of the original equation \eqref{3ode}, and under this basis, the monodromy matrices $N_j$ are given by
    \begin{equation}\label{eq-ss-21}N_j=\operatorname{diag}(1, \lambda_j^{-1}, \lambda_j),\quad j=1,2.\end{equation}
    Furthermore, $y_2(z)=y_1(-z)$ has the local exponent $-n-l$ at $0$.
\end{itemize}
\end{theorem}

\begin{proof}
(1) Suppose $Q_{n,l}(B)\neq 0$. Then Lemma \ref{lem-ode-4} implies that $\mathfrak{g}_1, \mathfrak{g}_2$ are linearly independent and so are $Y_1, Y_2$ by (\ref{eq-s-20}). Assume
\[a_0 Y_0(z)+a_1 Y_1(z)+a_2Y_2(z)=0\]
for some $a_0, a_1, a_2\in\mathbb{C}$ and $a_0\neq 0$. Then by $Y_0(z)=Y_0(-z)$ and $Y_2(z)=Y_1(-z)$ we have $a_2=a_1$.
Recalling Corollary \ref{coro-dual} that
\begin{equation}\label{eq-s-22}
Y_1(z+\omega_j)=\lambda_jY_1(z),\quad
Y_2(z+\omega_j)=\lambda_j^{-1}Y_2(z),\quad \;\, j=1,2,
\end{equation}
we deduce from $Y_0(z)$ being elliptic that
\[a_0 Y_0(z)+a_1 \lambda_jY_1(z)+a_2 \lambda_j^{-1}Y_2(z)=0,\]
and so $\lambda_j=1$ for $j=1,2$, i.e. $Y_1(z)$ and $Y_2(z)$ are also elliptic. Consequently, $Y_1(z)-Y_2(z)=Y_1(z)-Y_1(-z)$ is an odd elliptic solution of the dual equation (\ref{3ode-dual}) and so has the local exponent either $l+1$ or $n+l+2$ at $z=0$ by (\ref{eq-s-7}). But this implies that $Y_1(z)-Y_2(z)$ has no poles and vanishes at $z=0$, so $Y_1(z)-Y_2(z)=0$, a contradiction with that $Y_1, Y_2$ are linearly independent.

Therefore, $Y_0, Y_1, Y_2$ are linearly independent and (\ref{eq-s-21}) holds. If $\lambda_j=\pm 1$ for $j=1,2$, then $Y_1^2, Y_2^2$ are both elliptic and so are $\mathfrak{g}_1^2, \mathfrak{g}_2^2$. Since $\mathfrak{g}_1\mathfrak{g}_2, \mathfrak{g}_1^2, \mathfrak{g}_2^2$ form a basis of solutions of (\ref{eq-s-3}), we conclude that all solutions of (\ref{eq-s-3}) are elliptic, which is a contradiction with the fact that (\ref{eq-s-3}) has positive local exponents $1, 2\rho_2$ at $z=0$ as proved in Lemma \ref{lem-ode-3}. This proves $\lambda_j\neq \pm 1$ for at least one $j\in \{1,2\}$.

(2) By \eqref{eq: data} we know that $\hat{y}_0:=W(Y_1,Y_2)$,  $y_1:=W(Y_2, Y_0)$, $y_2:=W(Y_0, Y_1)$ are linearly independent solutions of the original equation (\ref{3ode-dual}) and
\begin{equation}\label{2eq-s-22}
y_1(z+\omega_j)=\lambda_j^{-1}y_1(z),\quad
y_2(z+\omega_j)=\lambda_jy_2(z),\quad \;\, j=1,2,
\end{equation}
namely the monodromy matrices with respect to $(\hat{y}_0, y_1, y_2)$ satisfy (\ref{eq-ss-21}). Since $Y_2(z)=Y_1(-z)$ implies that $\hat{y}_0(z)$ is an even elliptic solution of (\ref{3ode}), we conclude from Remark \ref{remark-evenelliptic} that $\hat{y}_0(z)=cy_0(z)$ for some constant $c\neq 0$.

Remark that $Y_2(z)=Y_1(-z)$ also implies $y_2(z)=y_1(-z)$, so it follows from \eqref{eq: data} that $W(y_1,y_2)$ is an even elliptic solution of the dual equation (\ref{3ode-dual}). This together with the assertion (1) implies $W(y_1, y_2)=cY_0$ for some constant $c\neq0$ and so has the local exponent $-n-2l$ at $0$. From here and the fact that the local exponent of $y_2(z)=y_1(-z)$ at $0$ is one of $-n-l, -l+1, n+2l+2$, we conclude that $y_2(z)=y_1(-z)$ has the local exponent $-n-l$ at $0$. This proves (2).
\end{proof}

The remaining case $Q_{n,l}(B)=0$ is quite different.

\begin{theorem}\label{Dual-Monodromy1}
If $Q_{n,l}(B)=0$, then $Y_1(z)=Y_2(z)$ and the monodromy matrices $M_1$, $M_2$ of the dual equation \eqref{3ode-dual} can not be diagonalized simultaneously. In particular, for each $j\in \{1,2\}$, the eigenvalues of $M_j$ are either $1,1,1$ or $1,-1,-1$.

Equivalently, the monodromy matrices $N_1$, $N_2$ of the original equation \eqref{3ode} can not be diagonalized simultaneously, and for each $j\in \{1,2\}$, the eigenvalues of $N_j$ are either $1,1,1$ or $1,-1,-1$.
\end{theorem}

\begin{proof}
Suppose $Q_{n,l}(B)=0$. Then Lemma \ref{lem-ode-4} and (\ref{eq-s-20}) imply that $Y_1, Y_2$ are linearly dependent, i.e. $Y_1(z)=cY_2(z)=cY_1(-z)$. So $c=\pm 1$ and then $Y_1(z)=Y_1(-z)=Y_2(z)$ because the local exponent of $Y_1$ at $z=0$ is even. Consequently, it follows from (\ref{eq-s-22}) that $\lambda_j=\lambda_j^{-1}=\pm 1$ for $j=1,2$.

{\bf Case 1.} $(\lambda_1, \lambda_2)\neq (1,1)$. Without loss of generality, we may assume $\lambda_1=-1$.

Then $Y_1$ is not elliptic and so linearly independent with $Y_0$. Since $\det M_j=1$, so the eigenvalues of $M_j$ are $1, \lambda_j, \lambda_j$, i.e. either $1,1,1$ or $1,-1,-1$ for $j=1,2$.

Suppose by contradiction that $M_1, M_2$ can be diagonized simultaneously, then (\ref{3ode-dual}) has a solution $\hat{Y}_2$ linearly independent with $Y_0, Y_1$ such that
\begin{equation}\label{eq-s-23}\hat{Y}_2(z+\omega_j)=\lambda_j\hat{Y}_2(z),\quad j=1,2.\end{equation}
Clearly $\hat{Y}_2(-z)$ is also a solution of the dual equation (\ref{3ode-dual}), so
\[\hat{Y}_2(-z)=c_0Y_0(z)+c_1Y_1(z)+c_2\hat{Y}_2(z)\]
for some constants $c_j$'s.
Consequently, $\lambda_1=-1$ gives
\begin{align*}&-c_0Y_0(z)-c_1Y_1(z)-c_2\hat{Y}_2(z)\\
=&-\hat{Y}_2(-z)
=\hat{Y}_2(-z+\omega_1)=c_0Y_0(z)-c_1Y_1(z)-c_2\hat{Y}_2(z),\end{align*}
which implies $c_0=0$. This, together with $Y_1(z)=Y_1(-z)$, leads to
\begin{align*}\hat{Y}_2(-z)&=c_1Y_1(z)+c_2\hat{Y}_2(z)\\
&=c_1Y_1(z)+c_2(c_1Y_1(-z)+c_2\hat{Y}_2(-z))\\
&=c_1(1+c_2)Y_1(z)+c_2^2\hat{Y}_2(-z),\end{align*}
namely $c_2=-1$. If $c_1=0$, then $\hat{Y}_2(z)=-\hat{Y}_2(-z)$ is odd and so its local exponent at $z=0$ is positive by (\ref{eq-s-7}), namely $\hat{Y}_2(z)$ has no poles and vanishes at $z=0$. Consequently, (\ref{eq-s-23}) implies that $\hat{Y}_2(z)$ is a bounded holomorphic function and so $\hat{Y}_2(z)=0$, a contradiction. In conclusion,
\[\hat{Y}_2(-z)=c_1Y_1(z)-\hat{Y}_2(z), \quad c_1\neq 0.\]

Recalling (\ref{eq-s-7}), we denote $d_o\in \{l+1, n+l+2\}$ to be the odd number. By the standard Forbenius' method, there is an \emph{odd} solution $Y_3(z)$ of the dual equation (\ref{3ode-dual}) which has the local exponent $d_o$ at $z=0$. We claim that
\begin{equation}\label{eq-s-24}
Y_3(z+\omega_j)=\lambda_jY_3(z),\quad j=1,2.\end{equation}

Indeed, there are constants $\tilde{c}_j$'s such that
\[Y_3(z)=\tilde{c}_0Y_0(z)+\tilde{c}_1Y_1(z)+\tilde{c}_2\hat{Y}_2(z).\]
Consequently,
\begin{align*}-Y_3(z)&=Y_3(-z)=\tilde{c}_0Y_0(z)+\tilde{c}_1Y_1(z)+\tilde{c}_2\hat{Y}_2(-z)\\
&=\tilde{c}_0Y_0(z)+(\tilde{c}_1+\tilde{c}_2c_1)Y_1(z)-\tilde{c}_2\hat{Y}_2(z),\end{align*}
so $\tilde{c}_0=0$ and $\tilde{c}_1+\tilde{c}_2c_1=-\tilde{c}_1$. This shows
$Y_3(z)=\tilde{c}_1Y_1(z)+\tilde{c}_2\hat{Y}_2(z)$
and so (\ref{eq-s-24}) holds. Then the same argument as before shows that $Y_3(z)$ is a bounded holomorphic function vanishing at $z=0$ and so $Y_3(z)=0$, a contradiction.
Therefore, $M_1, M_2$ can not be diagonalized simultaneously.

{\bf Case 2.} $(\lambda_1, \lambda_2)=(1,1)$.

Then $Y_1(z)$ is even elliptic and so $Y_0(z)=c_0Y_1(z)$ for some constant $c_0$ by applying Theorem \ref{even-elliptic-Y1}. In particular, (\ref{eq-s-10}) implies
\begin{equation}\label{eq-s-27}
Y_0(z)=c_0Y_1(z)=W(c_0\tilde{y}_1, y_0).
\end{equation}

Since $\det M_j=1$, we may assume that the eigenvalues of $M_j$ are $1, \varepsilon_j, \varepsilon_j^{-1}$ for $j=1,2$.

Suppose by contradiction that $M_1, M_2$ can be diagonized simultaneously, then (\ref{3ode-dual}) has solutions $\hat{Y}_1, \hat{Y}_2$ such that $Y_0, \hat{Y}_1, \hat{Y}_2$ are linearly independent and
\begin{equation}\label{eq-sss-23}\hat{Y}_1(z+\omega_j)=\varepsilon_j\hat{Y}_1(z),\quad j=1,2.\end{equation}
\begin{equation}\label{eq-ssss-23}\hat{Y}_2(z+\omega_j)=\varepsilon_j^{-1}\hat{Y}_2(z),\quad j=1,2.\end{equation}

{\bf Case 2-1.} $\varepsilon_j=\varepsilon_j^{-1}=\pm 1$ for $j=1,2$.

If
$(\varepsilon_1, \varepsilon_2)=(1,1)$, then all solutions of the dual equation (\ref{3ode-dual}) are elliptic, clearly a contradiction with (\ref{eq-s-7}). Therefore, $(\varepsilon_1, \varepsilon_2)\neq (1,1)$. Note that $\hat{Y}_1$ can not be odd, otherwise $\hat{Y}_1$ is a bounded holomorphic function vanishing at $z=0$ and so $\hat{Y}_1=0$, a contradiction. So by replacing $\hat{Y}_1(z)$ by $\hat{Y}_1(z)+\hat{Y}_1(-z)$ if necessarily, we have $\hat{Y}_1(z)=\hat{Y}_1(-z)$. Then the same argument as Case 1 leads to a contradiction. Thus Case 2-1 is impossible.

{\bf Case 2-2} $\varepsilon_j\neq \pm 1$ for some $j$, say $\varepsilon_1\neq \pm 1$ for example.

Then since $\hat{Y}_1(-z)$ also satisfies (\ref{eq-ssss-23}), we obtain from $\varepsilon_1\neq \pm 1$ that $\hat{Y}_2(z)=\hat{Y}_1(-z)$ up to multiplying a constant, so we may assume $\hat{Y}_2(z)=\hat{Y}_1(-z)$.
Consequently, the same argument as Theorem \ref{Dual-Monodromy}-(2) shows that  $y_0$,  $\hat{y}_1:=W(\hat{Y}_2, Y_0)$, $\hat{y}_2(z):=W(Y_0, \hat{Y}_1)=\hat{y}_1(-z)$ are linearly independent solutions of the original equation (\ref{3ode}) and
\[\hat{y}_1(z+\omega_j)=\varepsilon_j^{-1}\hat{y}_1(z),\quad
\hat{y}_2(z+\omega_j)=\varepsilon_j\hat{y}_2(z),\quad j=1,2.\]
Applying this argument again, we see that $Y_0$, $\tilde{Y}_1:=W(\hat{y}_2, y_0)$, $\tilde{Y}_2:=W(y_0,\hat{y}_1)=-W(\hat{y}_1, y_0)$ are linearly independent solutions of the dual equation (\ref{3ode-dual}). However, since (\ref{eq-s-27}) says $Y_0=W(c_0\tilde{y}_1, y_0)$ and $\tilde{y}_1$ is a linear combination of $y_0$,  $\hat{y}_1$, $\hat{y}_2$, it follows from $W(y_0,y_0)=0$ that $Y_0$ is a linear combination of $\tilde{Y}_1$ and $\tilde{Y}_2$, again a contradiction. So Case 2-2 is also impossible.

Therefore, $M_1, M_2$ can not be diagonalized simultaneously. Together this with the fact $M_1M_2=M_2M_1$, we conclude that for $j=1,2$, $\varepsilon_j=\varepsilon_j^{-1}$, namely the eigenvalues of $M_j$ are either $1,1,1$ or $1,-1,-1$.

The proof is complete.
\end{proof}

\begin{proof}[Proof of Theorem \ref{thm-ode-mono}]
Theorem \ref{thm-ode-mono} follows directly from Theorems \ref{Dual-Monodromy}-\ref{Dual-Monodromy1}.
\end{proof}

We conclude this section by the following simple observation.
Let $\zeta(z)=\zeta(z;\tau):=-\int^{z}\wp(\xi;\tau)d\xi$
be the Weierstrass zeta function with two quasi-periods $\eta_{k}(\tau)$, $
k=1,2$:
\begin{equation}
\eta_{k}(\tau):=2\zeta(\tfrac{\omega_{k}}{2} ;\tau)=\zeta(z+\omega_{k} ;\tau)-\zeta(z;\tau),\quad k=1,2,
\label{40-2}
\end{equation}
and $\sigma(z)=\sigma(z;\tau):=\exp \int^{z}\zeta(\xi)d\xi$ be the Weierstrass sigma function. Notice that $\zeta(z)$ is an odd
meromorphic function with simple poles at $\mathbb{Z}+\mathbb{Z}\tau$ and $\sigma(z)$
is an odd holomorphic function with simple zeros at $\mathbb{Z}+\mathbb{Z}\tau$.

Since $y_0(z)$ is even elliptic with only a pole at $0$ with order $2k$, we have that up to multiplying a constant,
\begin{equation}\label{fc-y0}y_0(z)=\frac{\prod_{j=1}^{2k}\sigma(z-p_j)}{\sigma(z)^{2k}},\end{equation}
where
\[\{p_1,\cdots,p_{2k}\}=\{-p_1,\cdots,-p_{2k}\}\subset E_{\tau}\setminus \{0\}\]
denotes the set of zeros of $y_0(z)$ up to multiplicities, and we can always assume that $p_j$ lies in the fundamental parallelogram of $E_{\tau}$ centered at the origin for all $j$.

\begin{proposition}\label{lemma-3.11} Under the above notations, all the zeros of $y_0(z)$ satisfy $p_j\to 0$ for all $j$ as $B\to \infty$, and
\begin{equation}\label{2e1q-112}\lim_{B\to \infty}(I(z)-\alpha\wp(z)-B)=0,\end{equation}
where $I(z)$ is defined in (\ref{eq-s-2}).

Similarly, all zeros of $Y_0(z)$ converge to $0$ as $B\to \infty$.
\end{proposition}

\begin{proof} Since $p_j$ lies in the fundamental parallelogram of $E_{\tau}$ centered at $0$, up to a subsequence of $B\to \infty$, we may assume $p_j=p_j(B)\to p_{\infty,j}$ for all $j$.
Define
\begin{equation}\label{2e1q-113}
h(z):=\frac{y_0'(z)}{y_0(z)}=\sum_{j=1}^{2k}\zeta(z-p_j)-2k\zeta(z),
\end{equation}
which is an odd elliptic function. Since
\[\frac{y_0'''(z)}{y_0(z)}=h''+3hh'+h^3,\]
we see from $y_0'''-(\alpha\wp+B)y_0'+\beta\wp' y_0=0$ that
\begin{equation}\label{2e1q-111}
h''+3hh'+h^3+\beta\wp'-\alpha\wp h=Bh.
\end{equation}
Clearly as $B\to\infty$, since the LHS of (\ref{2e1q-111}) is uniformly bounded in any compact subset of $E_{\tau}\setminus\{0, p_{\infty,1},\cdots, p_{\infty,2k}\}$, it follows from the RHS of (\ref{2e1q-111}) that $h(z)\to 0$ in any compact subset of $E_{\tau}\setminus\{0, p_{\infty,1},\cdots, p_{\infty,2k}\}$ as $B\to\infty$, so we conclude from (\ref{2e1q-113}) that $p_{\infty, j}=0$ for all $j$. From here we also have $h'(z), h''(z)\to 0$ as $B\to\infty$, so (\ref{2e1q-111}) implies
$Bh(z)\to \beta \wp'(z)$ as $B\to \infty$. Recalling (\ref{eq-s-2}) that
\[I(z)-\alpha\wp-B=-\frac{3y_0''}{2y_0}+\frac{3}{4}\left(\frac{y_0'}{y_0}\right)^2
=-\frac{3}{4}(2h'+h^2),\]
we see that (\ref{2e1q-112}) holds.

The same argument applies to the dual equation and implies that all the zeros of $Y_0(z)$ converges to $0$ as $B\to\infty$.
\end{proof}

\section{The special case $n$ even and $l=1$}

\label{sec-neven}

In this section, we consider the special case $n\geq 0$ even and $l=1$ and prove Theorem \ref{thml=1}. Recall that (\ref{alpha-nl})-(\ref{beta-nl}) give $\alpha=(n+2)(n+3)$ and $\beta=0$, so (\ref{3ode}) becomes
\begin{equation}\label{lame-derivative}
y'''(z)=[(n+2)(n+3)\wp(z)+B]y'(z),
\end{equation}
which is related to the Lam\'{e} equation
\begin{equation}\label{lame}
y''(z)=[(n+2)(n+3)\wp(z)+B]y(z).
\end{equation}
Recall the Lam\'{e} polynomial $\ell_{n+2}(B)$ stated in Theorem \ref{thmA}.
Then we have the following result.

\begin{lemma}\label{23ode-Lame} Let $n\geq 0$ be even and recall the polynomial $Q_{n,1}(B)$ for \eqref{lame-derivative} constructed in Lemma \ref{lem-ode-4}. Then
\begin{equation}\label{Lamepoly1}
Q_{n,1}(B)=\ell_{n+2}(B).
\end{equation}
In particular, the monodromy of \eqref{lame-derivative} is unitary for some $B$ if and only if the monodromy of the Lam\'{e} equation \eqref{lame} is unitary for the same $B$.
\end{lemma}

\begin{remark} Remark that the polynomial $Q_{n,1}(B)$ is defined only for even $n$. If $n$ is odd, then the second statement of Lemma \ref{23ode-Lame} does not necessarily hold either, because Theorem \ref{thm-nonexistence} indicates that the monodromy of (\ref{lame-derivative}) can not be unitary for any $B$ and any $\tau$, but it follows from Theorem \ref{thmB} that the monodromy of (\ref{lame}) could be unitary for some $B, \tau$'s.
\end{remark}

\begin{proof}[Proof of Lemma \ref{23ode-Lame}]
Remark that for (\ref{lame-derivative}), its unique even elliptic solution $y_0(z)=1$. Furthermore, $k=\frac{1-1}{2}=0$ and $m=n+2$, so $\deg Q_{n,1}=2(n+2)+1=\deg \ell_{n+2}$.

{\bf Step 1.}
Let $Q_{n,1}(B)\neq 0$. We prove that $\ell_{n+2}(B)\neq 0$.

Since $Q_{n,1}(B)\neq 0$, it follows from Theorem \ref{Dual-Monodromy} that (\ref{lame-derivative}) has a basis of solutions $(1, y_1, y_2)$ such that $y_1(-z)=y_2(z)$ and
\[N_j=\text{diag}(1, \lambda_j, \lambda_j^{-1}),\quad j=1,2,\]
for some $\lambda_1,\lambda_2\in\mathbb{C}\setminus\{0\}$ with $\lambda_j\neq \pm 1$ for at least one $j\in \{1,2\}$.

Let $f_1=y_1'$ and $f_2=y_2'$. Then $f_1, f_2$ are linearly independent solutions of the Lam\'{e} equation (\ref{lame}) and satisfy
\[f_1(z+\omega_j)=\lambda_jf_1(z), \quad f_2(z+\omega_j)=\lambda_j^{-1}f_2(z),\quad j=1,2,\]
namely with respect to $f_1, f_2$, the monodromy matrices
\[\widetilde{N}_j=\text{diag}(\lambda_j, \lambda_j^{-1}), \quad j=1,2.\]
This implies $\ell_{n+2}(B)\neq 0$. Furthermore, if the monodromy of (\ref{lame-derivative}) is unitary (i.e. $|\lambda_1|=|\lambda_2|=1$), then so does the monodromy of (\ref{lame}).

{\bf Step 2.} Let $\ell_{n+2}(B)\neq 0$. We prove that $Q_{n,1}(B)\neq 0$.

Since $\ell_{n+2}(B)\neq 0$, then by \cite{CLW,Gesztesy-Weikard}, there are linearly independent solutions $f_1(z)$, $f_2(z)=f_1(-z)$ of (\ref{lame}) such that (\ref{Mono-1}) holds for some $\lambda_1, \lambda_2\in\mathbb{C}\setminus\{0\}$ with $\{\lambda_1,\lambda_2\}\not\subset\{\pm 1\}$, i.e.
\begin{equation}\label{y1-trans}f_1(z+\omega_1)=\lambda_1f_1(z), \quad f_1(z+\omega_2)=\lambda_2f_1(z).\end{equation}
Without loss of generality, we may assume $\lambda_1\neq \pm 1$.

Recall that all solutions of (\ref{lame-derivative}) are meromorphic.
Let $y_1(z)$ be a solution of (\ref{lame-derivative}) such that $y_1'(z)=f_1(z)$. Then $y_1(z)$ is  meromorphic\footnote{This fact might not hold if $n$ is odd. Note that the local exponents of the Lam\'{e} equation (\ref{lame}) at $0$ are $-n-2, n+3$, so the local exponent of $f_1(z)$ at $0$ is $-n-2$. If $n$ is odd and the monodromy of (\ref{lame}) is unitary, then $y_1(z)$ has a logarithmic singularity at $0$ (which is equivalent to Res$_{z=0}f_1(z)\neq 0$), so $y_1(z)$ is not meromorphic. Indeed, if $y_1(z)$ is meromorphic, then the same proof will imply that the monodromy of (\ref{lame-derivative}) is unitary, a contradiction with Theorem \ref{thm-nonexistence}.} and is unique up to adding a constant. Let $y_2(z):=-y_1(-z)$, then $y_2'(z)=f_1(-z)=f_2(z)$. Suppose there exist constants $c_j$'s such that
\[c_0+c_1y_1+c_2y_2=0,\]
then $c_1f_1+c_2f_2=0$, which implies $c_1=c_2=0$ and so does $c_0$. Thus $(1, y_1, y_2)$ is a basis of solutions of (\ref{lame-derivative}). Since $y_1(z+\omega_1)$ is also a solution, we have
\[y_1(z+\omega_1)=d_0+d_1y_1(z)+d_2y_2(z),\]
where $d_i$'s are constants. Comparing this with (\ref{y1-trans}), we have $d_1=\lambda_1\neq \pm1$ and $d_2=0$. Then by replacing $y_1$ with $y_1-\frac{d_0}{1-\lambda_1}$, we have
\[y_1(z+\omega_1)=\lambda_1y_1(z),\]
and so
\[y_2(z+\omega_1)=-y_1(-z-\omega_1)=\lambda_1^{-1}y_2(z),\]
namely with respect to $(1, y_1, y_2)$, we have $N_1=\text{diag}(1, \lambda_1, \lambda_1^{-1})$. By $N_1N_2=N_2N_1$ and $\lambda_1\neq \pm1$, we obtain that $N_2=\text{diag}(1, \tilde{\lambda}_2, \tilde{\lambda}_2^{-1})$, i.e. $y_1(z+\omega_2)=\tilde{\lambda}_2 y_1(z)$. Again by (\ref{y1-trans}), we have $\tilde{\lambda}_2=\lambda_2$, so $N_2=\text{diag}(1, \lambda_2, \lambda_2^{-1})$. This implies $Q_{n,1}(B)\neq 0$.  Furthermore, if the monodromy of (\ref{lame}) is unitary (i.e. $|\lambda_1|=|\lambda_2|=1$), then so does the monodromy of (\ref{lame-derivative}).

{\bf Step 3.}
By Steps 1 and 2, we obtain
\[Q_{n,1}(B)= 0 \Longleftrightarrow \ell_{n+2}(B)=0.\]
Since $\deg Q_{n,1}=2(n+2)+1=\deg \ell_{n+2}$ and it is well known (cf. \cite{Whittaker-Watson}) that the Lam\'{e} polynomial $\ell_{n+2}(\cdot)$ has $2(n+2)+1$ distinct roots for generic $\tau$, we conclude that \eqref{Lamepoly1} holds for generic $\tau$ and hence for all $\tau$ by continuity.
\end{proof}

\begin{proof}[Proof of Theorem \ref{thml=1}] Theorem \ref{thml=1} follows from Lemma \ref{23ode-Lame} and Theorem \ref{thmB}.
\end{proof}

\subsection*{Acknowledgements} The research of Z. Chen was supported by NSFC (No. 12222109, 12071240).

\subsection*{Conflict of interest} The authors have no conflicts to disclose.

\end{document}